\documentclass[sn-mathphys-num]{sn-jnl}

\usepackage{graphicx}%
\usepackage{multirow}%
\usepackage{amsmath,amssymb,amsfonts}%
\usepackage{amsthm}%
\usepackage{mathrsfs}%
\usepackage[title]{appendix}%
\usepackage{xcolor}%
\usepackage{textcomp}%
\usepackage{manyfoot}%
\usepackage{booktabs}%
\usepackage{algorithm}%
\usepackage{algorithmicx}%
\usepackage{algpseudocode}%
\usepackage{listings}%

\theoremstyle{thmstyleone}%
\newtheorem{theorem}{Theorem}[section]
\newtheorem{proposition}[theorem]{Proposition}%

\newtheorem{lemma}[theorem]{Lemma}
\newtheorem{example}[theorem]{Example}

\usepackage{lipsum}
\usepackage{amsfonts}
\usepackage{graphicx}
\usepackage{epstopdf}
\ifpdf
  \DeclareGraphicsExtensions{.eps,.pdf,.png,.jpg}
\else
  \DeclareGraphicsExtensions{.eps}
\fi
\usepackage{tabularx,booktabs}
\usepackage{color}
\usepackage{lscape}
\usepackage{latexsym}
\usepackage{rotating}
\usepackage{amsmath,amsfonts,amssymb,enumerate}
\usepackage{threeparttable}
\usepackage{graphicx}
\usepackage{multirow}
\usepackage{algpseudocode,algorithm}
\usepackage{bm}
\usepackage{galois}

\renewcommand{\thefigure}{\thesection.\arabic{figure}}

\renewcommand{\thetable}{\Roman{table}}
\counterwithout{equation}{section}

\DeclareMathOperator*{\prox}{prox}

\DeclareMathOperator*{\argmin}{arg\, min}

\newcommand{\mtrx}[1]{\mathsf{#1}}
\newcommand{\mA}{\mtrx{A}}
\newcommand{\mB}{\mtrx{B}}

\newcommand{\mH}{\mtrx{H}}
\newcommand{\mI}{\mtrx{\mathrm{I}}}

\newcommand{\mP}{\mtrx{P}}

\newcommand{\vect}[1]{\bm{#1}}

\newcommand{\vb}{\vect{b}}

\newcommand{\ve}{\vect{e}}

\newcommand{\vu}{\vect{u}}
\newcommand{\vv}{\vect{v}}
\newcommand{\vx}{\vect{x}}

\newcommand{\vz}{\vect{z}}
\newcommand{\vw}{\vect{w}}

\newcommand{\vzero}{\vect{0}}

\newcommand{\R}{\mathbb{R}}

\newcommand{\Ss}{\mathbb{S}}
\newcommand{\la}{\langle}
\newcommand{\ra}{\rangle}
\newcommand{\Rnd}{\R_{\downarrow}^n}
\newcommand{\mU}{\mtrx{U}}

\newcommand{\ul}{\underline}

\usepackage{algorithm}
\usepackage{algpseudocode}

\begin{document}

\title[Computing Proximity Operators of Scale and Signed Permutation Invariant Functions]{Computing Proximity Operators of Scale and Signed Permutation Invariant Functions}


\author[1]{\fnm{Jianqing} \sur{Jia}}\email{jjia10@syr.edu}
\equalcont{These authors contributed equally to this work.}

\author[2]{\fnm{Ashley} \sur{Prater-Bennette}}\email{ashley.prater-bennette@us.af.mil}

\author*[1]{\fnm{Lixin} \sur{Shen}}\email{lshen03@syr.edu}
\equalcont{These authors contributed equally to this work.}

\affil*[1]{\orgdiv{Department of Mathematics}, \orgname{Syracuse University}, \orgaddress{\city{Syracuse}, \postcode{NY 13244}, \country{USA}}}

\affil[2]{\orgdiv{Information Directorate}, \orgname{Air Force Research Laboratory}, \orgaddress{\city{Rome}, \postcode{NY 10587},\country{USA}}}



\abstract{This paper investigates the computation of proximity operators for scale and signed permutation invariant functions. A scale-invariant function remains unchanged under uniform scaling, while a signed permutation invariant function retains its structure despite permutations and sign changes applied to its input variables. Noteworthy examples include the $\ell_0$ function and the ratios of $\ell_1/\ell_2$  and its square, with their proximity operators being particularly crucial in sparse signal recovery. We delve into the properties of scale and signed permutation invariant functions, delineating the computation of their proximity operators into three sequential steps: the $\vw$-step, $r$-step, and $d$-step. These steps collectively form a procedure termed as WRD, with the $\vw$-step being of utmost importance and requiring careful treatment. Leveraging this procedure, we present a method for explicitly computing the proximity operator of $(\ell_1/\ell_2)^2$ and introduce an efficient algorithm for the proximity operator of $\ell_1/\ell_2$.}

\keywords{sparse promoting functions, proximity operator, $\ell_1/\ell_2$, $(\ell_1/\ell_2)^2$ }


\pacs[MSC Classification]{90C26, 90C32, 90C55, 90C90, 65K05}

\maketitle

\section{Introduction}
This paper addresses the computation of the proximity operator for scale and signed permutation invariant functions. A scale-invariant function is characterized by its resilience to uniform scaling: it remains unaltered when its input undergoes a constant factor multiplication. This invariance extends to permutations, ensuring that changes in the order of input variables do not affect the function's value. Additionally, the function exhibits invariance under sign changes, meaning that if any component of an input is replaced by its negative counterpart, the function value remains consistent. In the context of this study, a signed permutation invariant function is defined as a mathematical function that retains its form despite permutations and sign changes applied to its input variables.

Several well-known examples of signed permutation invariant functions, as well as scale and signed permutation invariant functions, are presented:
\begin{itemize}
\item All $\ell_p$ norms, where $0< p \le \infty$, and log-sum penalty function in $\mathbb{R}^n$ are signed permutation invariant but not scale invariant, see \cite{Candes-Wakin-Boyd:JFAA:08,Prater-Shen-Tripp:JSC:2022};
\item The $\ell_0$ norm and the effective sparsity measure $\left(\frac{\|\cdot\|_q}{\|\cdot\|_1}\right)^{\frac{q}{1-q}}$,  $q\in (0, \infty) \setminus \{1\}$ are both scale and signed permutation invariant, see \cite{Lopes:IEEEIT:2016,Rahimi-Wang-Dong-Lou:SIAMSC:2019,Tang-Nehorai:IEEESP:2011,Yin-Esser-Xin:CIS14,Xu-Narayan-Tran-Webster:ACHA2021}.
\end{itemize}

The proximity operator is a mathematical concept used in optimization. This operator provides a computationally efficient way to find solutions for optimization problems involving nonsmooth functions \cite{Attouch-Bolte-Svaiter:MP:13,Beck-Teboulle:SIAMIS:09,Bolte-Sabach-Teboulle:MP:2014,Combettes-Wajs:MMS:05,Krol-Li-Shen-Xu:IP:2012,Li-Shen-Xu-Zhang:AiCM:15,Micchelli-Shen-Xu:IP-11,Parikh-Boyd:NF-Opt:14}. Given a proper lower semicontinuous function $f$ and a point $\vx$, the proximity operator of $f$ at $\vx$, denoted as $\mathrm{prox}_f(\vx)$, is defined as:
$$
\mathrm{prox}_f(\vx)=\arg\min\left\{\frac{1}{2}\|\vu-\vx\|_2^2+f(\vu): \vu \in \mathbb{R}^n\right\}.
$$
In simpler terms, the proximity operator finds a point $\vu$ that minimizes the sum of the function $f$ and half of the squared Euclidean distance between $\vu$ and a given point $\vx$.

Our focus of this paper is to study the proximity operator of scale and signed permutation invariant functions.
Our approach for computing the proximity operator of scale and signed permutation invariant functions is based on this observation: the space $\mathbb{R}^n$ is isomorphic to the Cartesian product of $\mathbb{R}$ and the $(n-1)$ dimensional unit sphere, denoted by $\mathbb{S}^{n-1}$. Mathematically, this can be expressed as:
$$
\mathbb{R}^n \cong \mathbb{R} \times \mathbb{S}^{n-1}.
$$
That is, for $\vu \in \mathbb{R}^n$, it can be converted to a pair $(r, \vw) \in \mathbb{R} \times \mathbb{S}^{n-1}$ such that
$\vu = r \vw$,  where  $r=\|\vu\|_2$ and $\vw=\vu/\|\vu\|_2$.
With this conversion, the task of finding a point $\vu \in \mathrm{prox}_f(\vx)$ transforms into finding a pair of $(r, \vw) \in \mathbb{R} \times \mathbb{S}^{n-1}$ such that $\vu = r \vw$. Exploring the properties of the scale and signed permutation invariant functions $f$, the process of finding this pair $(r, \vw)$ involves three consecutive steps. The first step is to solve an optimization problem with variable $\vw$ only, the second step straightforwardly yields $r=\langle \vx, \vw\rangle$, and the final step involves deciding whether to choose the origin or the scaled vector $\vu=r \vw$ as the resulting point. Clearly, the first step is crucial.

For all scale and signed permutation invariant functions, we will present a complete study on the following function
\begin{equation*}
h_p(\vx)=\left(\frac{\|\vx\|_1}{\|\vx\|_2}\right)^p \quad \mbox{for} \quad p=1,2.
\end{equation*}
Notably, there has been a gap in existing literature concerning the proximity operator of $h_2$, and we have observed a recent study that addresses the proximity operator of $h_1$ \cite{Tao:SIAMSC-2022}. In our work, we aim to fill this gap by providing a comprehensive analysis of the proximity operator for both $h_1$ and $h_2$ within the context of scale and signed permutation invariant functions.

With our approach, the optimization problem for $\vw$ associated with both $h_1$ and $h_2$ is nonconvex and takes the form of a constrained quadratic programming problem after certain simplifications. Despite the nonconvex nature of the objective functions and the constrained sets, we adopt a distinct strategy to address them individually.

For the $h_2$ function, the objective function of the quadratic programming problem involves only a quadratic term formulated by a structured symmetric rank-2 matrix. Explicitly demonstrating that this matrix possesses one positive eigenvalue and one negative eigenvalue, and the constrained set of the problem is $\mathbb{S}^{n-1}\cap \mathbb{R}^n_{+}$, where $\mathbb{R}^n_{+}$ is the first orthant of $\mathbb{R}^n$. While both the objective function and constrained set are nonconvex, we are able to develop a procedure to find the optimal solution $\vw$ through the eigenvector of the matrix corresponding to the negative eigenvalue, achieved in a finite number of iterations.

For the $h_1$ function, the objective function of the quadratic programming problem comprises a quadratic term formulated by a rank-one symmetric matrix and one linear term. The rank-1 matrix is negative definite, and the constrained set remains  $\mathbb{S}^{n-1}\cap \mathbb{R}^n_{+}$. Similar to the situation with $h_2$, both the objective function and constrained set are nonconvex. However, the procedure utilized for $h_2$ cannot be directly adapted for $h_1$. To address this, we relax the nonconvex feasible set $\mathbb{S}^{n-1}\cap \mathbb{R}^n_{+}$  to a convex set $\{\vw \in \mathbb{R}^n_{+}: \|\vw\|_2\le 1\}$. The resulting optimization problem maintains the same objective function as the non-relaxed version, but is now constrained in a convex domain. We establish conditions ensuring that the optimal solution to the relaxed problem lies on $\mathbb{S}^{n-1}\cap \mathbb{R}^n_{+}$ or to be the origin. Subsequently, we propose a projected gradient method to solve the relaxed optimization problem. Leveraging the fact that the optimal solution is related to the proximity operator of $h_1$ at a given point, we use this information as prior knowledge to initialize the projected gradient method. Through numerical experiments, our findings consistently indicate that the algorithm can successfully find the optimal $\vw$ for the original, unrelaxed optimization problem.

It's worth noting that a different approach for the proximity operator of $h_1$ has been reported recently in \cite{Tao:SIAMSC-2022}. That paper claimed to have derived the analytical solution of the proximity operator of $h_1$, relying on prior knowledge about the sparsity of the corresponding output from this proximity operator, which, however, is unknown in general. A bisection method was then applied for finding this desired sparsity.

The current literature, including works such as \cite{Lopes:IEEEIT:2016, Rahimi-Wang-Dong-Lou:SIAMSC:2019, Tang-Nehorai:IEEESP:2011, Yin-Esser-Xin:CIS14}, suggests that both $h_1$ and $h_2$  functions can effectively promote sparsity in underlying signals. However, to the best of our knowledge, there is a lack of theoretical justification for this claim. In this paper, we provide the theoretical proof that both $h_1$ and $h_2$ functions qualify as sparsity-promoting functions, as defined in \cite{Shen-Suter-Tripp:JOTA:2019}.

The outline of the rest of the paper is as follows: In Section~\ref{sec:properties}, we begin by presenting some properties of the proximity operators for scale and signed permutation invariant functions. These properties allow us to focus our discussion on these proximity operators within a specific set: each point lies in the first orthant of $\mathbb{R}^n$, and the entries of the point are in descending order. By employing a different representation of the points in this set, determining the proximity operators of scale and signed permutation invariant functions at these points essentially reduces to solving a quadratic programming problem constrained on a nonconvex set. We then introduce a comprehensive procedure called the WRD procedure, which comprises three distinct steps: $\vw$-step, $r$-step and $d$-step. This procedure enables efficient computation of proximity operators for scale and signed permutation invariant functions, offering a systematic approach to solving such problems.

In Section~\ref{sec:h2}, utilizing the WRD procedure, we compute the proximity operator of $h_2$. We are able to provide an explicit solution for the proximity operator of $h_2$  at any point in a highly efficient manner, thereby demonstrating the effectiveness of our approach.

In Section~\ref{sec:h1}, leveraging the WRD procedure, we compute the proximity operator of $h_1$. We are able to develop an efficient algorithm to evaluate the proximity operator of $h_1$ at any point, showcasing the versatility of our methodology.

The conclusion of this paper is drawn in Section~\ref{sec:conclusions}, summarizing the findings and contributions of our study. We discuss the implications of our results and propose avenues for future research.

\section{Scale and Signed Permutation Invariant Functions and their Proximity Operators}\label{sec:properties}

All functions in this work are defined on $\mathbb{R}^n$  the Euclidean space of dimension $n$. Bold lowercase letters, such as $\vx$, signify vectors, with the $j$th component represented by the corresponding lowercase letter $x_j$.  Matrices are indicated by bold uppercase letters such as  $\mA$ and $\mB$.  We use $\mathbb{R}^n_{+}$ to denote the set of points in $\mathbb{R}^n$ such that all entries of each point in the set are nonnegative. The cone of vectors $\vx$ in $\mathbb{R}^n_{+}$ satisfying $x_1 \ge x_2 \ge \ldots \ge x_n \ge 0$ is denoted by $\mathbb{R}^n_\downarrow$. We use $\mathbb{S}^{n-1}$ (or $\mathbb{B}^{n}$) to denote the unit sphere (or ball) in $\mathbb{R}^{n}$. We use $\mathbb{S}^{n-1}_{+}$ ($\mathbb{B}^{n}_{+}$ or $\mathbb{B}^{n}_{\downarrow}$) to denote the partial unit sphere $\mathbb{S}^{n-1}\cap \mathbb{R}^n_{+}$ (the partial unit ball $\mathbb{B}^{n}\cap \mathbb{R}^n_{+}$ or $\mathbb{B}^{n}\cap \mathbb{R}^n_{\downarrow}$) in $\mathbb{R}^{n}$.  Let $\mathcal{P}_n$ denote the set of all $n\times n$ signed permutation matrices: those matrices that have only one nonzero entry in every row or column, which is $\pm1$.



The $\ell_p$ norm of $\vx=[x_1,\ldots, x_n]^\top \in \mathbb{R}^n$ is defined as $\|\vx\|_p=(\sum_{k=1}^n |x_k|^p)^{1/p}$ for $1\le p <\infty$ and $\|\vx\|_\infty=\max\{|x_k|: k=1,2,\ldots,n\}$. When $p=0$,
$\|\vx\|_0$ represents the number of non-zero components in $\vx$. The standard inner product in $\mathbb{R}^n$ is denoted by $\langle \vu,\vv\rangle$, where $\vu$ and $\vv$ are vectors in $\mathbb{R}^n$.

We denote $[n]:=\{1,2,\ldots,n\}$ as an index set up to a positive integer $n$. For a subset $S$ of $[n]$, the notation $|S|$ represents the cardinality of $S$. For a vector $\vx \in \mathbb{R}^n$ and a subset $S$ of $[n]$,  $\vx_{S}$ denotes the vector that retains the entries with indices in $S$ of $\vx$ and sets the remaining entries to zero, or the subvector of $\vx$ with indices solely from $S$. The specific meaning of $\vx_S$ being referred to will be evident from the context of the discussion.

A function $f: \mathbb{R}^n \rightarrow \mathbb{R}$  is considered scale invariant if for all $\vx\in \mathbb{R}^n$ and $\alpha>0$, the following holds:
$$
f(\alpha \vx) = f(\vx).
$$
In other words, scaling the input by any positive constant does not alter the value of the function.

A function $f: \mathbb{R}^n \rightarrow \mathbb{R}$ is considered signed permutation invariant if it remains unchanged under the action of permutations and sign changes of its input variables. Formally, a function $f$ is signed permutation invariant if, for all permutations $\mP \in \mathcal{P}_n$ and for all vectors $\vx \in \mathbb{R}^n$, the following holds:
$$
f(\mP\vx) = f(\vx).
$$

A function $f$ defined on $\mathbb{R}^n$ with values in $\mathbb{R} \cup \{+\infty\}$ is proper if its domain $\mathrm{dom}(f)=\{x\in \mathbb{R}^n: f(x)<+\infty\}$ is nonempty, and $f$ is lower semicontinuous if its epigraph is a closed set. The set of proper and lower semicontinuous functions on  $\mathbb{R}^n$ to $\mathbb{R} \cup \{+\infty\}$ is denoted by $\Gamma(\mathbb{R}^n)$.

The proximity operator was introduced by Moreau in \cite{moreau:RASPS:62}. For a function $f \in \Gamma(\mathbb{R}^n)$, the proximity operator of $f$ at $\vz \in  \mathbb{R}^n$ with index $\alpha$ is defined by
$$
\mathrm{prox}_{\alpha f} (\vz) := \mathrm{arg}\min \left\{\frac{1}{2\alpha} \|\vx-\vz\|_2^2+f(\vx): \vx\in\mathbb{R}^n\right\}.
$$
The proximity operator of $f$ is a set-valued operator from $\mathbb{R}^n \rightarrow 2^{\mathbb{R}^n}$, the power set of $\mathbb{R}^n$. In this paper, for a scale and signed permutation function, we always assume that the set $\mathrm{prox}_{\alpha f} (\vz)$ is nonempty and compact.


\subsection{Properties}
The proximity operator exhibits certain properties concerning scale and signed permutation invariant functions.
\begin{lemma}\label{lem:properties}
Let $\vx \in \mathbb{R}^n$, $\mP \in \mathcal{P}_n$, $\alpha>0$, and $\lambda>0$. The following statements hold:
\begin{itemize}
\item[(i)] For a signed permutation invariant function $f \in \Gamma(\mathbb{R}^n)$, $\mathrm{prox}_{\lambda f}(\vx) = \mP^{-1} \mathrm{prox}_{\lambda f} (\mP\vx)$.
\item[(ii)] For a scale invariant function $f \in \Gamma(\mathbb{R}^n)$, $\mathrm{prox}_{\lambda f}(\alpha \vx) = \alpha \mathrm{prox}_{\lambda \alpha^{-2} f} (\vx)$.
\end{itemize}
\end{lemma}
\begin{proof}\ \ The proof of the two items is based on the definitions of the proximity operator and scale and signed permutation invariant function. We skip the details of the proof here.
\end{proof}
%


For any vector $\vx \in \mathbb{R}^n$, there is a signed permutation $\mP \in \mathcal{P}_n$ such that $\mP \vx \in  \mathbb{R}_{\downarrow}^n$, that is, the entries of $\vx$ can be sorted  in a way of $|x_{\sigma(1)}|\ge |x_{\sigma(2)}| \ge \ldots \ge |x_{\sigma(n)}|$, where $\sigma(i)$ is the index of nonzero entry in the $i$th column of $\mP$.  By Lemma~\ref{lem:properties}, for a signed permutation invariant function in $\Gamma(\mathbb{R}^n)$, it is sufficient to consider its proximity operator on  $\mathbb{R}_{\downarrow}^n$.

For a vector $\vx \in \mathbb{R}_{\downarrow}^n$, we assert that $\vx$ exhibits $k$ blocks, characterized by  $(k+1)$ distinct indices $\{i_j: j\in [k+1]\}$ satisfying  $i_1=1$, $i_{k+1}=n$, and $i_{j}<i_{j+1}$. In these blocks, $\vx$ follows the pattern  $x_{i_j}=x_{i_{j+1}-1}<x_{i_{j+1}}$ for $1\le j \le  k-1$ and $x_{i_{k}}=x_{i_{k+1}}$. In essence, the vector $\vx$ comprises $k$ blocks, where entries within each block are identical, yet they differ from entries in other blocks.

\begin{lemma}\label{lemma:prox}
Let $f$ be a signed permutation invariant function in $\Gamma(\mathbb{R}^n)$, and let $\lambda>0$. Consider $\vx \in \mathbb{R}^n_{\downarrow}$, we assert that $\prox_{\lambda f}(\vx) \subseteq \mathbb{R}^n_{+}$. Furthermore, there exists a point $\vu \in \prox_{\lambda f}(\vx)$ such that $\vu \in \mathbb{R}_{\downarrow}^n$.
\end{lemma}
\begin{proof} \ \ To establish $\prox_{\lambda f}(\vx) \subseteq \mathbb{R}^n_{+}$, we observe that the objective function from the definition of $\prox_{\lambda f}(\vx)$ is
$\frac{1}{2\lambda}\|\vu-\vx\|_2^2+f(\vu)$ for all $\vu \in \mathbb{R}^n$. As $f$ is a signed permutation invariant function, $f(\vu)=f(\mP\vu)$ for all $\mP \in \mathcal{P}_n$. Given $\vx \in \mathbb{R}_{\downarrow}^n$, our discussion can be restricted to $\vu \in \R_{+}^n$; otherwise, say the first element $u_1$ of $\vu$ is negative, then $(-u_1-x_1)^2+\sum_{\ell=2}^n (u_\ell-x_\ell)^2 \le (u_1-x_1)^2+\sum_{\ell=2}^n (u_\ell-x_\ell)^2$. From the above discussion, we conclude that $\prox_{\lambda f}(\vx) \subseteq \mathbb{R}_{+}^n$.

Now, suppose $\vu \in \prox_{\lambda f}(\vx)$. If the vector $\vx$ has one block, that is, all entries of $\vx$ are the same. Clearly, we can rearrange entries of $\vu$ so that the rearranged one is in $\mathbb{R}^n_\downarrow$ and is still in $\prox_{\lambda f}(\vx)$. If vector $\vx \in \R_{\downarrow}^n$ has $k\ge 2$ blocks, characterized by $(k+1)$ distinct indices $\{i_j: j\in [k+1]\}$. We define $u_{\overline{j}}=\max\{u_\ell: i_j\le \ell \le i_{j+1}-1\}$ and $u_{\underline{j}}=\min\{u_\ell: i_j\le \ell \le i_{j+1}-1\}$ for $1 \le j \le k-1$, and $u_{\overline{k}}=\max\{u_\ell: i_k\le \ell \le i_{k+1}\}$ and $u_{\underline{k}}=\min\{u_\ell: i_k\le \ell \le i_{k+1}\}$. We claim that $u_{\underline{j}} \ge u_{\overline{j+1}}$ for $1 \le j \le k-1$. If these inequalities do not hold for some $1 \le i \le k-1$, assume, without loss of generality, that $u_{\underline{1}} < u_{\overline{2}}$. One can assume that $u_1=u_{\underline{1}}$ and $u_{i_2}=u_{\overline{2}}$.  In this case, let $\widetilde{\vu}$ be a vector from $\vu$ by exchanging its first and the $i_2$ components. Immediately, $f(\widetilde{\vu})=f(\vu)$, and
\begin{eqnarray*}
\|\widetilde{\vu}-\vx\|_2^2-\|\vu-\vx\|_2^2
&=&(u_{\overline{2}}-x_1)^2+(u_{\underline{1}}-x_{i_2})^2-(u_{\underline{1}}-x_1)^2-(u_{\overline{2}}-x_{i_2})^2\\
&=&2(u_{\underline{1}}-u_{\overline{2}})(x_1-x_{i_2})< 0
\end{eqnarray*}
due to the conditions of $x_1=x_{i_1}>x_{i_2}$ and $u_{\underline{1}} < u_{\overline{2}}$. This conflicts with our assumption of  $\vu \in \prox_{\lambda f}(\vx)$.

Finally, since all entries in each block of $\vx$ are the same, arranging the entries of $\vu \in \prox_{\lambda f}(\vx)$ for the indices in the same block in descending order results in $\vu$ still belonging to $\prox_{\lambda f}(\vx)$. Thus, there exists a point $\vu \in \prox_{\lambda f}(\vx)$ such that $\vu \in \R_{\downarrow}^n$.
\end{proof}

\subsection{Reformulation}
Our focus of this paper is to study the proximity operator of scale and signed permutation invariant functions. Our approach for computing the proximity operator of scale and signed permutation invariant functions is based on this observation: the space $\mathbb{R}^n$ is isomorphic to the Cartesian product of $\mathbb{R}$ and $\mathbb{S}^{n-1}$.
That is, for $\vu \in \mathbb{R}^n$, it can be converted to a pair $(r, \vw) \in \mathbb{R} \times \mathbb{S}^{n-1}$ such that
$$
\vu = r \vw,
$$
where
$$
r=\|\vu\|_2 \in \mathbb{R} \quad \mbox{and} \quad  \vw=\frac{\vu}{\|\vu\|_2} \in \mathbb{S}^{n-1}.
$$
With this conversion, the task of finding a point $\vu \in \mathrm{prox}_f(\vx)$ transforms into finding a pair of $(r, \vw) \in \mathbb{R} \times \mathbb{S}^{n-1}$ such that $\vu = r \vw$.

\begin{theorem}\label{thm:main}
    Let $f$ be a scale and signed permutation invariant function in $\Gamma(\mathbb{R}^n)$, and let $\rho>0$. Consider a vector $\vx \in \mathbb{R}_{\downarrow}^n$ and define
    \begin{equation}\label{def:F}
        F(\vu):=\frac{\rho}{2}\|\vu-\vx\|_2^2 +  f (\vu).
    \end{equation}
Then $\vx^\star \in \mathrm{prox}_{\frac{1}{\rho}f}(\vx)$ if and only if $\vx^\star$ is given by
\begin{equation}\label{x-star}
\vx^{\star} \in \left\{
\begin{array}{ll}
    \{\mathbf{0}\}, & \hbox{if $F(\mathbf{0})<F(\langle \vx, \vw^\star\rangle \vw^\star)$;} \\
    \{\mathbf{0},\langle \vx, \vw^\star\rangle \vw^\star\}, & \hbox{if $F(\mathbf{0})=F(\langle \vx, \vw^\star\rangle \vw^\star)$;} \\
    \{\langle \vx, \vw^\star\rangle \vw^\star\}, & \hbox{otherwise.}
  \end{array}
\right.
\end{equation}
where $\vw^\star$ is a solution to the following optimization problem
    \begin{equation}\label{w-star}
        \min\left\{-\frac{\rho}{2}\langle \vx, \vw\rangle^2+f(\vw): \vw \in \mathbb{S}^{n-1}_{+}\right\}.
    \end{equation}
\end{theorem}
\begin{proof}
From the definition of proximity operator,
$$
\mathrm{prox}_{\frac{1}{\rho}f}(\vx)=\argmin\left\{F(\vu): \vu \in \mathbb{R}^n\right\}.
$$
By Lemma~\ref{lem:properties} and Lemma~\ref{lemma:prox}, for $\vx \in \mathbb{R}_{\downarrow}^n$ we establish that
$$
\argmin\left\{F(\vu): \vu \in \mathbb{R}^n\right\}=\argmin\left\{F(\vu): \vu \in \mathbb{R}^n_{+}\right\}.
$$
To delve deeper into the optimization problem on the right-hand side, we express  $\vu = r \vw$ with $r\ge 0$ and $\vw \in \mathbb{S}_{+}^{n-1}$. Consequently, for $r=0$,
$$
F(\mathbf{0})=\frac{\rho}{2}\|\vx\|_2^2+f(\mathbf{0})$$
and for $r>0$
\begin{eqnarray}
F(\vu) &=& \frac{\rho}{2}\|r\vw-\vx\|_2^2 +  f (r\vw) \nonumber \\
&=&\frac{\rho}{2}(r^2-2r\langle \vw, \vx\rangle +\|\vx\|_2^2) + f(\vw)  \nonumber \\
&=&\frac{\rho}{2}(r-\langle \vw, \vx\rangle)^2+\frac{\rho}{2}\|\vx\|_2^2+\left(-\frac{\rho}{2}\langle \vw, \vx\rangle^2 + f(\vw)\right). \label{tmp:equi}
\end{eqnarray}
In equation \eqref{tmp:equi}, the terms are as follows: The first term $\frac{\rho}{2}(r-\langle \vw, \vx\rangle)^2$ can always achieve the minimum value $0$ by taking $r=\langle \vw, \vx\rangle$; the second term $\frac{\rho}{2}\|\vx\|_2^2$ is constant with respect to the pair $(r,\vw)$; and third term $-\frac{\rho}{2}\langle \vw, \vx\rangle^2 + f(\vw)$  is solely a function of $\vw$. Therefore, we seek $\vw^\star$ that minimizes the third term with respect to $\vw$, i.e., solving the optimization problem~\eqref{w-star}, then form the expression $\langle \vx, \vw^\star\rangle \vw^\star$. Hence, the conclusion of this theorem holds.
\end{proof}

In the following discussion, we use the notation $F$ in \eqref{def:F} to represent the objective function for $\mathrm{prox}_{\frac{1}{\rho}f}(\vx)$ and denote
    \begin{equation}\label{def:G}
        G(\vw):=-\frac{\rho}{2}\langle \vx, \vw\rangle^2+f(\vw).
    \end{equation}
to represent the objective function of \eqref{w-star}.

The significance of the scale and signed permutation invariance of $f$ becomes evident in the proof of the theorem above.  The scale invariance of $f$ facilitates the discussion from $\mathbb{R}^{n}$ to $\mathbb{S}^{n-1}$, while the signed permutation invariance narrows the focus from $\mathbb{S}^{n-1}$ to $\mathbb{S}^{n-1}_{+}$, allowing us to isolate the impact of $r$ and $\vw$ when solving an optimization problem that involves $\vw$ exclusively.

In accordance with Theorem~\ref{thm:main}, the process of determining the pair $(r,\vw)$ involves three distinct steps:
\begin{itemize}
\item $\vw$-step: In this step, the objective is to find an optimal solution $\vw^\star$ to the optimization problem~\eqref{w-star}.
\item $r$-step: Following the $\vw$-step, the corresponding $r^\star$ is computed as $r^\star=\langle \vx, \vw^\star\rangle$, where $\vw^\star$ is the output from {$\vw$-step}.
\item $d$-step: This final step determines $\vx^\star$ according to \eqref{x-star}.
\end{itemize}
Upon completing these three steps, as shown in \eqref{x-star}, $\vx^\star$ belongs to $\mathrm{prox}_{\frac{1}{\rho}f}(\vx)$. For ease of reference in the subsequent discussion, this procedure is referred to as WRD ($\vw$-step, $r$-step, $d$-step).

To show the applicability of the WRD procedure, we present the proximity operator of the $\ell_0$ norm, a typical scale and signed permutation invariant function.


\begin{example}\label{example:Haar}
The proximity operator of the $\ell_0$ norm at $\vx$ with index $1/\rho$ is, see, e.g.,  \cite{donoho:ieeeit:95,Shen-Suter-Tripp:JOTA:2019},
$$
(\mathrm{prox}_{\frac{1}{\rho}\|\cdot\|_0}(\vx))_i=
\left\{
\begin{array}{ll}
    \{x_i\}, & \hbox{if $|x_i|>\sqrt{2/\rho}$;} \\
    \{0,x_i\}, & \hbox{if $|x_i|=\sqrt{2/\rho}$;}\\
    \{0\}, & \hbox{otherwise.}
  \end{array}
\right.
$$

We intend to apply the WRD procedure for computing $\mathrm{prox}_{\frac{1}{\rho}\|\cdot\|_0}$. Assuming $\vx \in \mathbb{R}^n_{\downarrow}$, and following the approach used in the proof of Theorem~\ref{thm:main}, we define $F(\vu):=\frac{\rho}{2}\|\vu-\vx\|_2^2 +  \|\vu\|_0$. The next step involves seeking  the optimal solution to optimization problem~\eqref{w-star} for $\vw \in  \mathbb{S}^{n-1}_{+}$, where $G(\vw):=-\frac{\rho}{2}\langle \vx, \vw\rangle^2 +\|\vw\|_0$.  Thus, for $\vw \in  \mathbb{S}^{n-1}_{+}$ with an $\ell_0$ norm of $k$, the smallest value of $G$ is achieved when $\vw$ is aligned with the first $k$ entries of $\vx$, that is,
$$
G\left(\frac{\vx_{[k]}}{\|\vx_{[k]}\|_2}\right)=-\frac{\rho}{2} \|\vx_{[k]}\|_2^2+k=\sum_{i=1}^k \left(-\frac{\rho}{2} x_i^2+1\right).
$$
Here $\vx_{[k]}$ keeps the first $k$ entries of $\vx$ and sets the remaining entries zeros. Therefore, the output in the $\vw$-step of the WRD procedure is given by
$$
\argmin_{\vw \in \mathbb{S}^{n-1}_{+}} G(\vw)=\left\{
\begin{array}{ll}
    \left\{\frac{\vx_{[1]}}{\|\vx_{[1]}\|_2}\right\}, & \hbox{if $x_1<\sqrt{2/\rho}$;} \\

    \left\{\frac{\vx_{S}}{\|\vx_{S}\|_2}: S \subseteq [p], |S|\ge 1\right\}, & \hbox{if $\exists p \in [n]$ s.t. $x_1=x_p=\sqrt{2/\rho}>x_{p+1}$;}\\

    \left\{\frac{\vx_{[k]}}{\|\vx_{[k]}\|_2}\right\}, & \hbox{if $\exists k \in [n]$ s.t. $x_k>\sqrt{2/\rho}>x_{k+1}$;}\\

    \left\{\frac{\vx_{[k]\cup S}}{\|\vx_{[k]\cup S}\|_2}: S \subseteq [p], |S|\ge 1 \right\}, & \hbox{if $\exists k \in [n]$ and $p\in[n-k]$ s.t. }\\
    &\mbox{ \hskip 0.2cm $x_k>\sqrt{2/\rho}=x_{k+1}=x_{k+p}>x_{k+p+1}$}. \\

  \end{array}
\right.
$$
This output represents the solutions to the $\vw$-step of the WRD. Subsequently, choosing a vector  $\vw^\star \in \argmin_{\vw \in \mathbb{S}^{n-1}_{+}} G(\vw)$, the $r$-step generates $r=\langle \vx, \vw^\star\rangle$.  With the pair $(r,\vw^\star)$, the $d$-step of the  WRD compares the difference between $F(\langle \vx, \vw^\star\rangle \vw^\star)$ and  $F(\mathbf{0})$, resulting in
$$
F(\langle \vx, \vw^\star\rangle \vw^\star) - F(\mathbf{0}) = \sum_{i=1}^k \left(-\frac{\rho}{2} x_i^2+1\right),
$$
where $k$ is equal to $1$ if $x_1 \le \sqrt{2/\rho}$ or is the integer such that $x_k>\sqrt{2/\rho} \ge x_{k+1}$. Clearly,  the vector $\mathbf{0}$ is in $\mathrm{prox}_{\frac{1}{\rho}\|\cdot\|_0}(\vx)$ if $x_1 < \sqrt{2/\rho}$, and both $\mathbf{0}$ and $\langle \vx, \vw^\star\rangle \vw^\star$ are in $\mathrm{prox}_{\frac{1}{\rho}\|\cdot\|_0}(\vx)$ if $x_1 = \sqrt{2/\rho}$; otherwise  $\langle \vx, \vw^\star\rangle \vw^\star$ is in $\mathrm{prox}_{\frac{1}{\rho}\|\cdot\|_0}(\vx)$. These discussions affirm that the WRD procedure accurately recovers the proximity operator of the $\ell_0$ norm.
\end{example}

In the rest of the paper, we focus on computing the proximity operator of the function below:
\begin{equation}\label{def:hp}
h_p(\vx)=\left\{
\begin{array}{ll}
    \left(\frac{\|\vx\|_1}{\|\vx\|_2}\right)^p, & \hbox{if $\vx \neq \mathbf{0}$;} \\
    0,&{otherwise},
  \end{array}
\right.
\end{equation}
for $p=1$ and $2$. This function is lower semicontinuous and for all nonzero vectors $\vx \in \mathbb{R}^n$,
$1\le h_p(\vx) \le n^{p/2}$. Thus, the proximity operator of $h_p$ at any point is nonempty.
Notably, setting the value of $h_p$ at the origin to any value smaller than or equal to 1 preserves the lower semicontinuity of the function. For example, $h_1(\mathbf{0})$ is set to be $1$ as illustrated in \cite{Tao:SIAMSC-2022}.  Therefore, our proposed WRD procedure remains applicable. Lastly, it's important to note that in $\mathbb{R}$, our definition of $h_p$ aligns consistently with the $\ell_0$ norm, that is, $h_p(\vx)=\|\vx\|_0$ for $\vx \in \mathbb{R}$.

In the next section, we consider the computation of the proximity operator of $h_2$ first.

\section{The Proximity Operator of $h_2$}\label{sec:h2}
We plan to use the WRD procedure to compute the proximity operator of $h_2$. We begin with showing the optimization problem \eqref{w-star} associated with the $\vw$-step of the WRD.

Define $\ve$ as a vector with all its components $1$. For $\vx \in \mathbb{R}_{+}^n$, we have
$$\langle \vw, \vx\rangle^2 = \vw^\top \vx \vx^\top \vw \quad \mbox{and} \quad \|\vw\|_1^2 = \vw^\top \ve\ve^\top\vw.
$$
Set
\begin{equation}\label{def:A2}
\mA_{\rho,\vx} := 2\ve\ve^\top-{\rho}\vx\vx^\top.
\end{equation}
The corresponding function $G$ in \eqref{def:G} becomes
$$
G(\vw)=\frac{1}{2}\vw^\top \mA_{\rho,\vx} \vw,
$$
Hence, the optimization problem \eqref{w-star} is a quadratic programming constrained on $\mathbb{S}_+^{n-1}$.

We promptly obtain a result concerning the proximity operator of $h_2$ at points that are multiples of the vector $\ve$ as follows:
\begin{theorem}\label{thm:positive-multiple-1}
For $\rho>0$ and  $\vx =\alpha \ve$ for some $\alpha>0$, then
$$
\mathrm{prox}_{\frac{1}{\rho} h_2} (\vx)=\left\{
                          \begin{array}{ll}
                            \{\alpha \ve\}, & \hbox{if $\rho \alpha^2>2$;} \\
                            \{\mathbf{0}\}\cup\{\alpha \|\vw\|_1 \vw: \vw \in \mathbb{S}_{+}^{n-1}\}, & \hbox{if $\rho \alpha^2=2$;} \\
                            \{\mathbf{0}\}, & \hbox{if $\rho \alpha^2<2$.}
                          \end{array}
                        \right.
$$
\end{theorem}
\begin{proof}\ \ In the situation of $\vx =\alpha \ve$ for some $\alpha>0$, we have $\mA_{\rho,\vx}=(2-\rho \alpha^2) \ve\ve^\top$ from \eqref{def:A2}. The objective function of problem~\eqref{w-star} is $G(\vw)= \frac{1}{2} (2-\rho \alpha^2) \|\vw\|_1^2$. To investigate the minimal value of the above function on $\mathbb{S}_{+}^{n-1}$ and at which point the optimal is achieved, there are three different situations according to the value of $\rho \alpha^2$.

If $\rho \alpha^2>2$,  the minimal value of $\frac{1}{2} \vw^\top \mA_{\rho,\vx} \vw$ is achieved at $\vw$ which has the largest $\ell_1$ norm for $\vw \in \mathbb{S}_{+}^{n-1}$. Clearly, the optimal $\vw^\star$ must be $\frac{1}{\sqrt{n}}\ve$ and $G(\vw^\star)=\frac{1}{2} (2-\rho \alpha^2)n<0$.  Hence, $\mathrm{prox}_{\frac{1}{\rho} h_2} (\vx)=\{\langle \alpha \ve, \vw^\star \rangle \vw^\star\} = \{\alpha \ve\}$.

If $\rho \alpha^2=2$, then $G(\vw)=0$ for all $\vw \in \mathbb{S}^{n-1}_{+}$. Note that $\langle \alpha \ve, \vw \rangle \vw=\alpha \|\vw\|_1\vw$. Hence, $\mathrm{prox}_{\frac{1}{\rho} h_2} (\vx)=\{\mathbf{0}\}\cup \{\alpha \|\vw\|_1 \vw: \vw \in \mathbb{S}_{+}^{n-1}\}$.

Finally, if $\rho \alpha^2<2$, then the minimal value of $G$ on $\mathbb{S}_{+}^{n-1}$ is achieved at $\vw^\star \in \{\ve_i: i\in [n]\}$ and $G(\ve_i)=\frac{1}{2} (2-\rho \alpha^2)>0$ for all $i \in [n]$. Hence, $\mathrm{prox}_{\frac{1}{\rho} h_2} (\vx)=\{\mathbf{0}\}$.
\end{proof}


By Lemma~\ref{lem:properties}, we restrict our attention to the proximity operator of $h_2$ on $ \mathbb{R}_{\downarrow}^n$. The complete discussion is presented in the following two subsections. In the first subsection, we conduct a comprehensive analysis of the proximity operator of $h_2$ specially in $\mathbb{R}^2$. We delve into the intricacies of this operator, exploring its behavior and characteristics within this constrained domain. In the second subsection, we begin with investigating the properties of the eigenvectors of the matrix  $\mA_{\rho,\vx}$. The eigenvector corresponding to a negative eigenvalue plays a pivotal role in determining the solution in the $\vw$-step of the WRD procedure. By leveraging these properties effectively, we explicitly derive the proximity operator of $h_2$ over the entire space $\mathbb{R}^n$.

\subsection{Special case: the proximity operator of $h_2$ on $\mathbb{R}^2$}
The following result is about the proximity operator of $h_2$ on $\mathbb{R}^2_\downarrow$.
\begin{theorem}\label{thm:n=2}
For $\rho>0$ and  $\vx \in \mathbb{R}_{\downarrow}^2$ not a multiple of $\ve$, write
$$
\theta^\star= \left\{
               \begin{array}{ll}
                 \frac{1}{2}\arctan \left(\frac{-2(2-\rho x_1x_2)}{\rho(x_1^2-x_2^2)}\right), & \hbox{if $\rho x_1x_2 > 2$;} \\
                 0, & \hbox{if $\rho x_1x_2 \le 2$.}
               \end{array}
             \right.
$$
then, $\vw^\star=\begin{bmatrix}\cos \theta^\star & \sin \theta^\star\end{bmatrix}^\top$ is the optimal solution to problem \eqref{w-star}. Finally,
$$
\mathrm{prox}_{\frac{1}{\rho} h_2}(\vx)=\left\{
               \begin{array}{ll}
               \{\mathbf{0}\}, & \hbox{if $\rho x_1x_2 \le 2$ and $\rho x_1^2<2$;} \\
               \{\mathbf{0},x_1 \ve_1\}, & \hbox{if $\rho x_1x_2 \le 2$ and $\rho x_1^2=2$;} \\
               \{x_1 \ve_1\}, & \hbox{if $\rho x_1x_2 \le 2$ and $\rho x_1^2>2$;} \\
                \{\langle \vx, \vw^\star\rangle \vw^\star\}, & \hbox{if $\rho x_1x_2 > 2$.}
               \end{array}
             \right.
$$
\end{theorem}
\begin{proof}\ \ For $\vw \in \mathbb{S}^1_{+}$, we have
$$
G(\vw)=\frac{1}{2}(2-\rho x_1^2)w_1^2+\frac{1}{2}(2-\rho x_2^2)w_2^2+(2-\rho x_1x_2) w_1w_2.
$$
Write $w_1=\cos \theta$ and $w_2=\sin \theta$. The function $G$ can be written as
$$
G(\vw) =\frac{1}{2}(2-\rho x_2^2)-\frac{\rho}{4}(x_1^2-x_2^2)+\underbrace{\frac{\rho}{4}\rho(x_2^2-x_1^2)\cos(2 \theta)+\frac{1}{2}(2-\rho x_1 x_2)\sin (2\theta)}_{Q(\theta):=}.
$$
It is clear that minimizing $\vw^\top \mA_{\rho, \vx}\vw$ for $\vw \in \mathbb{S}^1_{+}$ is equivalent to minimizing the function $F(\theta)$ for $\theta \in [0, \pi/2]$.  By Lemma~\ref{lem:properties} and Theorem~\ref{thm:main}, we can restrict the parameter $\theta \in [0, \pi/4]$.

To investigate the global minimizer of $Q$ over the interval $\theta \in [0, \pi/4]$, we compute the derivative of $Q$ as follows
$$
Q'(\theta)=\frac{1}{2}\rho(x_1^2-x_2^2)\sin(2 \theta)+(2-\rho x_1 x_2)\cos (2\theta).
$$
We consider two cases. Case 1: If $2-\rho x_1 x_2\ge 0$, $Q'(\theta) \ge 0$ for $\theta \in [0, \pi/4]$. Hence, $Q$ achieves its global minimum at $\theta=0$. That is, $\vw^\star=\ve_1$. Case 2: If $2-\rho x_1 x_2<0$, $Q'$ has only one root $\theta^\star$ on $[0, \pi/4]$, given by $\theta^\star=\frac{1}{2}\arctan \left(\frac{-2(2-\rho x_1x_2)}{\rho(x_1^2-x_2^2)}\right)$. Due to $Q'(0)=2(2-\rho x_1 x_2)<0$ and $Q'(\pi/4)=\rho(x_1^2 - x_2^2)>0$. Hence, $Q$ achieves its global minimum at $\theta^\star$.  As a result, $\vw^\star=\begin{bmatrix}\cos \theta^\star & \sin \theta^\star\end{bmatrix}^\top$ is the optimal solution to problem \eqref{w-star}. This completes the $\vw$-step of the WRD procedure for $\mathrm{prox}_{\frac{1}{\rho} h_2}(\vx)$. The $r$-step follows immediately with $r^\star=\langle \vx, \vw^\star\rangle$.

Finally, for the $d$-step of the WRD procedure, we only need to know the sign of $G(\vw^\star)$. For Case 1, $G(\vw^\star)=G(\ve_1)=\frac{1}{2}(2-\rho x_1^2)$, which is positive if $\rho x_1^2<2$, zero if $\rho x_1^2=2$, and negative otherwise.  For Case 2, $G(\vw^\star)<G(\ve_1)=\frac{1}{2}(2-\rho x_1^2)<\frac{1}{2}(2-\rho x_1x_2)<0$. So, from the sign of  $G(\vw^\star)$, we conclude  $\mathrm{prox}_{\frac{1}{\rho} h_2}(\vx)$.
\end{proof}

To close this subsection, a detailed examination of the proximity operator of $h_2$ with index $1/\rho$ in $\mathbb{R}^2$ is conducted through visual representation via plots. In addition, the proximity operator of the $\ell_0$ norm with index $1/\rho$ is incorporated for comparative analysis, considering $h_2$ as an approximation of the $\ell_0$ norm. The ensuing visualizations aim to provide insights into the behavior and characteristics of the proximity operator for $h_2$ in comparison to the $\ell_0$ norm, enhancing our understanding of their respective properties in $\mathbb{R}^2$. As stipulated by Lemma~\ref{lem:properties}, we exclusively present the behavior on $\mathbb{R}^2_{\downarrow}$.

Figure~\ref{fig:Haar2-S-D}(a) illustrates the proximity operator of the $\ell_0$ norm. Following the guidance from  Example~\ref{example:Haar}, the set $\mathbb{R}^2_{\downarrow}$ is divided into  three distinct regions I, II, and III as depicted in Figure~\ref{fig:Haar2-S-D}(a) and defined as follows:
\begin{eqnarray*}
\mbox{Region I}&=&\{(x_1,x_2): 0\le x_2 \le x_1\le \sqrt{{2}/{\rho}}\},    \\
\mbox{Region II}&=&\{(x_1,x_2): 0\le x_2 \le \sqrt{{2}/{\rho}}<x_1\},    \\
\mbox{Region III}&=&\{(x_1,x_2): \sqrt{{2}/{\rho}}< x_2 \le x_1\}.
\end{eqnarray*}
On Region I, the $\mathrm{prox}_{\frac{1}{\rho} \|\cdot\|_0}$ at the corner $\sqrt{{2}/{\rho}}\ve$ is $\{\mathbf{0}, \sqrt{{2}/{\rho}}\ve, \sqrt{{2}/{\rho}}\ve_1, \sqrt{{2}/{\rho}}\ve_2\}$;  at each other point on the line $x_1=\sqrt{{2}/{\rho}}$, it is $\sqrt{{2}/{\rho}}\ve_1$; and at each other point in Region I, it is $\mathbf{0}$. On Region II,  $\mathrm{prox}_{\frac{1}{\rho}\|\cdot\|_0}$ at the point $(x_1, \sqrt{{2}/{\rho}})$ is $\{(x_1, \sqrt{{2}/{\rho}}), x_1\ve_1\}$ and at each other point $(x_1,x_2)$ is $x_1\ve_1$. On Region III, $\mathrm{prox}_{\frac{1}{\rho} \|\cdot\|_0}$ at each point is itself.

Figure~\ref{fig:Haar2-S-D}(b) showcases the proximity operator of the $h_2$ on the line $x_1=x_2$. The operator $\mathrm{prox}_{\frac{1}{\rho} h_2}$ at each point $\alpha\ve$ is $\mathbf{0}$ if  $\alpha<\sqrt{2/\rho}$ (blue dash-dot line);  $\{\mathbf{0}\}\cup\{\alpha \|\vw\|_1 \vw: \vw \in \mathbb{S}_{+}^{n-1}\}$ if $\alpha=\sqrt{2/\rho}$ (marked by the square); and $\alpha \ve$ itself if $\alpha>\sqrt{2/\rho}$ (magenta dot line).  Comparing with the proximity operator of the $\ell_0$ norm, the main difference is at the point $\sqrt{2/\rho}\ve$.

Figure~\ref{fig:Haar2-S-D}(c) exhibits the proximity operator of the $h_2$ on $\mathbb{R}^2_{\downarrow}$ excluding the line $x_1=x_2$. The set $\mathbb{R}^2_{\downarrow}$ partitions into three regions I, II, and III as shown in Figure~\ref{fig:Haar2-S-D}(c) and defined as follows:
\begin{eqnarray*}
\mbox{Region I}&=&\{(x_1,x_2): 0\le x_2 < x_1\le \sqrt{{2}/{\rho}}\},    \\
\mbox{Region II}&=&\{(x_1,x_2): 0\le x_2 \le {2}/{(\rho x_1)}, x_1>\sqrt{{2}/{\rho}}\},    \\
\mbox{Region III}&=&\{(x_1,x_2): {2}/{(\rho x_1)}< x_2 \le x_1, x_1>\sqrt{{2}/{\rho}}\}.
\end{eqnarray*}
On Region I, the $\mathrm{prox}_{\frac{1}{\rho} h_2}$ at each  point on the line $x_1=\sqrt{{2}/{\rho}}$ is $\{\mathbf{0},\sqrt{{2}/{\rho}}\ve_1\}$; the $\mathrm{prox}_{\frac{1}{\rho} h_2}$ at each  other point
is $\mathbf{0}$. On Region II, the $\mathrm{prox}_{\frac{1}{\rho} h_2}$ at each  point $(x_1,x_2)$ is $x_1\ve_1$ (see the red line). On Region III, the $\mathrm{prox}_{\frac{1}{\rho} h_2}$ at each  point $\vx$ is $\langle \vx, \vw^\star\rangle \vw^\star$, where $\vw^\star$ is given in Theorem~\ref{thm:n=2}. Specifically, results for three lines with their slopes 0.9 (green line), 0.5 (cyan line), and 0.3 (black line) are presented, and the $\mathrm{prox}_{\frac{1}{\rho} h_2}$ at these points are represented by dashed lines with corresponding colors.

\begin{figure}[ht]
\centering
\resizebox{\textwidth}{!}
{
\begin{tabular}{ccc}
\includegraphics[width=1.6in]{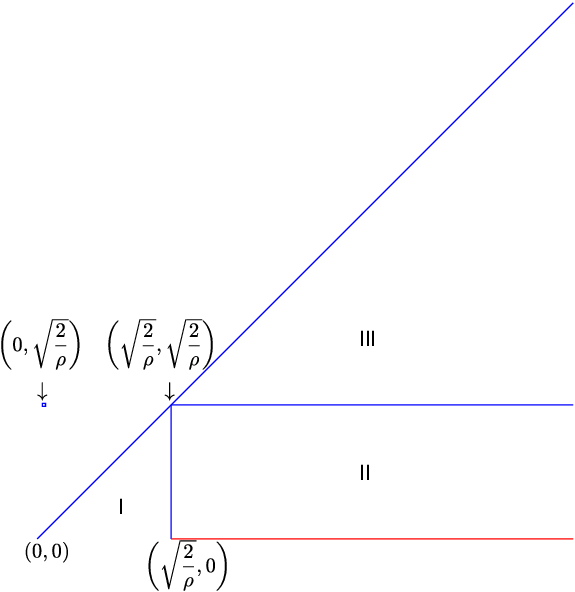}&
\includegraphics[width=1.6in]{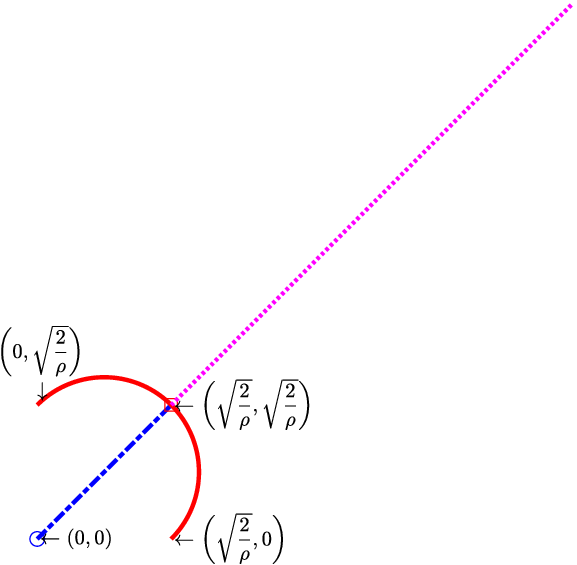}&
\includegraphics[width=1.6in]{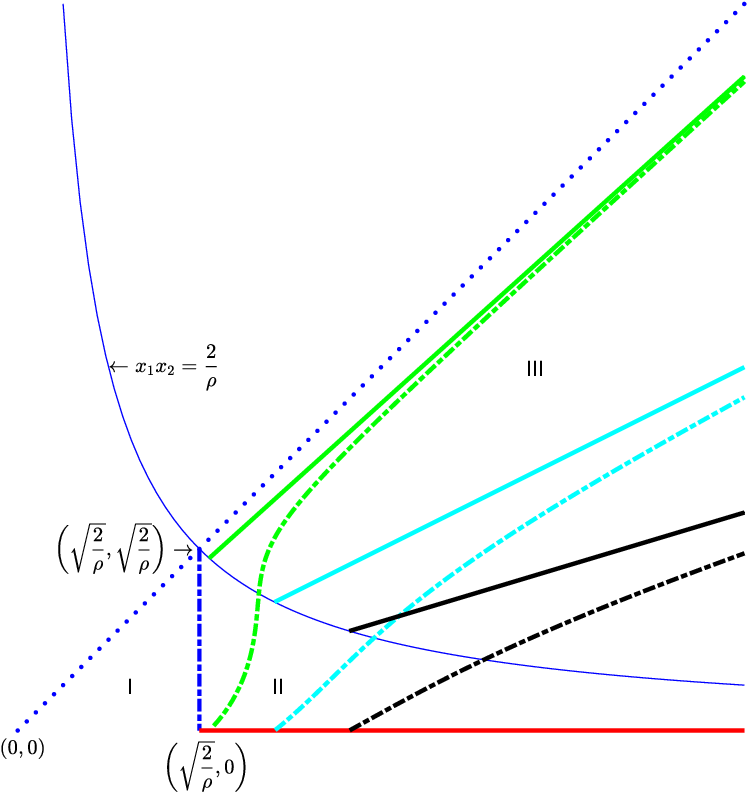}\\

(a)&(b)&(c)
\end{tabular}
}
\caption{The plots of the proximity operator in $\mathbb{R}^2_{\downarrow}$ for (a) the $\ell_0$ norm; (b) $h_2$ on the line with the slope $1$; and (c) $h_2$ on $\mathbb{R}^2_{\downarrow}$ excluding the line  with the slope $1$.}
\label{fig:Haar2-S-D}
\end{figure}

\subsection{General case: the proximity operator of $h_2$ on $\mathbb{R}^n$}

In the preceding subsection, we explored the determination of the proximity operator of  $h_2$ on $\mathbb{R}^2$ through the WRD procedure. The central concept involved parameterizing $\mathbb{S}^{1}_{+}$ using a single variable, simplifying the resulting problem in the $\vw$-step of the WRD procedure and facilitating ease of solution.  While $\mathbb{S}^{n-1}_{+}$ for $n>2$ can be parameterized by $(n-1)$ parameters, the ensuing problem in the $\vw$-step appears to be intricate for direct analysis. Consequently, alternative approaches must be considered to address and overcome the complexities associated with this scenario.

Given the pivotal role of the $\vw$-step in the WRD procedure, this subsection places particular emphasis on this phase. It is noteworthy that the objective function $G$ for the $\vw$-step is characterized as a quadratic form. In this context, we invoke the following two pertinent results.

\begin{lemma}[Theorem 1 in \cite{tao1996difference}] \label{lemma:Tao-An}
Consider the following optimization problem
\begin{equation} \label{eq:tao}
    \min\left\{\frac{1}{2}\vw^\top \mH \vw + \vb^\top \vw: \|\vw\|_2=r\right\},
\end{equation}
where $\mH$ is an $n \times n$ symmetric matrix, $\vb \in \mathbb{R}^n$ and $r$ a positive number. A vector $\vw^\star$ is a solution to this problem if and only if there is a real number $\lambda^\star$ such that (i) $\mH+\lambda^\star \mI$ is positive semi-definite; (ii) $(\mH+\lambda^\star \mI)\vw^\star=-\vb$; and $\|\vw^\star\|_2=r$. Such a $\lambda^\star$ is unique.
\end{lemma}

\begin{lemma}[\cite{martinez1994local,tao1996difference}] \label{lemma:Tao-An2}
Consider the optimization problem \eqref{eq:tao}. If $\vb$ is orthogonal to some eigenvector associated with the smallest eigenvalue, then there is no local-nonglobal minimum for \eqref{eq:tao}.
\end{lemma}

Note that both Lemma~\ref{lemma:Tao-An} and Lemma~\ref{lemma:Tao-An2} consider the quadratic optimization problems constrained on a sphere. However, our problem in $\vw$-step is restricted on $\mathbb{S}^{n-1}_{+}$.

To investigate the applicability of Lemma~\ref{lemma:Tao-An} for the optimization problem in the $\vw$-step, a crucial prerequisite is understanding the eigen-structure of the matrix $\mA_{\rho,\vx}$, as defined in \eqref{def:A2}. This matrix is the sum of two rank-1 matrices;  consequently, it possesses at most two non-zero eigenvalues. In order to delve into the eigen-structure of the matrix  $\mA_{\rho,\vx}$, let's introduce a set of notations:
\begin{eqnarray}
\Delta&:=&\left(\frac{\rho}{2}\|\vx\|_2^2+n\right)^2-2\rho \|\vx\|_1^2,\label{eq:Delta} \\
\underline{\alpha}&:=&\left(\frac{\rho}{2}\|\vx\|_2^2+n\right)-\sqrt{\Delta},  \label{eq:alpha-under} \\
\overline{\alpha}&:=&\left(\frac{\rho}{2}\|\vx\|_2^2+n\right)+\sqrt{\Delta}. \label{eq:alpha-over}\\
\underline{\lambda}&:=&2n-\overline{\alpha} \label{eq:lambda-under}\\
\overline{\lambda}&:=&2n-\underline{\alpha}  \label{eq:lambda-over}\\
\underline{\vw}&:=&\vx-\frac{\underline{\alpha}}{\rho \|\vx\|_1}\ve \label{eq:u-under}\\
\overline{\vw}&:=&\vx-\frac{\overline{\alpha}}{\rho \|\vx\|_1}\ve.  \label{eq:u-over}
\end{eqnarray}
Observations about the above notations are as follows: The inequality $\|\vx\|_1 \le \sqrt{n}\|\vx\|_2$ implies that $\Delta \ge \left(\frac{\rho}{2}\|\vx\|_2^2-n\right)^2$, and the inequality strictly holds if $\vx$ is not a multiple of $\ve$. This observation further implies that both $\underline{\alpha}$ and $\overline{\alpha}$ (given in \eqref{eq:alpha-under} and \eqref{eq:alpha-over}) are non-negative numbers. For $\overline{\lambda}$ (given in \eqref{eq:lambda-under}):
$$
\overline{\lambda}=2n-\underline{\alpha}= -\left(\frac{\rho}{2}\|\vx\|_2^2-n\right)+\sqrt{\Delta}\ge  -\left(\frac{\rho}{2}\|\vx\|_2^2-n\right)+\left|\frac{\rho}{2}\|\vx\|_2^2-n\right|\ge 0,
$$
where the equality holds if $\vx$ is a multiple of $\ve$. Similarly, for $\underline{\lambda}$ (given in \eqref{eq:lambda-over}):
$$
\underline{\lambda}=2n-\overline{\alpha}= -\left(\frac{\rho}{2}\|\vx\|_2^2-n\right) -\sqrt{\Delta} \le  -\left(\frac{\rho}{2}\|\vx\|_2^2-n\right)-\left|\frac{\rho}{2}\|\vx\|_2^2-n\right|\le 0,
$$
where the equality holds if $\vx$ is a multiple of $\ve$ again. Hence, if $\vx$ is not a multiple of $\ve$, then $\overline{\lambda}$ is positive, while $\underline{\lambda}$ is negative.

The subsequent result elucidates the eigenstructure of the matrix $\mA_{\rho, \vx}$.
\begin{proposition}\label{prop:A} Let $\mA_{\rho,\vx}$ be given in \eqref{def:A2} for $\rho>0$ and  $\vx \in \R_{+}^n$. Let $\underline{\alpha}$, $\overline{\alpha}$, $\underline{\lambda}$, $\overline{\lambda}$,  $\underline{\vw}$, and $\overline{\vw}$ be given by \eqref{eq:alpha-under}, \eqref{eq:alpha-over},  \eqref{eq:lambda-under}, \eqref{eq:lambda-over},   \eqref{eq:u-under}, and \eqref{eq:u-over}, respectively.  The following statements hold:
\begin{itemize}
\item[(i)] Assume $\vx =\alpha \ve$ for some $\alpha>0$. Then the matrix $\mA_{\rho,\vx}$ has only zero as its eigenvalues if $\rho \alpha^2=2$; or has $(2-\rho \alpha^2)n$ as its only non-zero eigenvalue with $\frac{1}{\sqrt{n}}\ve$ the corresponding eigenvector.
\item[(ii)] Assume $\vx \neq \alpha \ve$ for any $\alpha$. Then the matrix $\mA_{\rho,\vx}$ has only one positive eigenvalue $\overline{\lambda}$ and one negative eigenvalue $\underline{\lambda}$ given as
$\overline{\lambda}=2n-\underline{\alpha}$ and $\underline{\lambda}=2n-\overline{\alpha}$. The corresponding eigenvectors associated with $\overline{\lambda}$ and $\underline{\lambda}$ are
$\overline{\vw}=\vx-\frac{\overline{\alpha}}{\rho \|\vx\|_1}\ve$ and $\underline{\vw}=\vx-\frac{\underline{\alpha}}{\rho \|\vx\|_1}\ve$, respectively.
\end{itemize}
\end{proposition}
\begin{proof}\ \
Item (i). In this case, $\mA_{\rho,\vx}=\left(2-\rho\alpha^2\right) \ve\ve^\top$. Clearly, $\mA_{\rho,\vx}=\mathbf{0}$ if $\rho\alpha^2=2$, so $\mA_{\rho,\vx}$ has only zero as its eigenvalues. Otherwise, $\mA_{\rho,\vx}$ has $\left(2-\rho\alpha^2\right)n$ as its only non-zero eigenvalue with the corresponding eigenvector $\frac{1}{\sqrt{n}}\ve$.

Item (ii). From $\mathrm{rank}(\mA_{\rho,\vx}) \le \mathrm{rank}(2\ve\ve^\top)+\mathrm{rank}({\rho}\vx\vx^\top)=2$, the matrix $\mA_{\rho,\vx}$ has at most two nonzero eigenvalues which will be found as follows.
For any $\lambda$, a direct computation gives
$$
\mA_{\rho,\vx}(\vx-\lambda\ve)=-{\rho}(\|\vx\|_2^2-\|\vx\|_1\lambda)\vx+2(\|\vx\|_1-n\lambda)\ve.
$$
If the vector $\vx-\lambda\ve$ is the eigenvector of $\mA_{\rho,\vx}$, then the equation
$$
-{\rho}(\|\vx\|_2^2-\|\vx\|_1\lambda)=\frac{2(\|\vx\|_1-n\lambda)}{-\lambda}
$$
holds for some $\lambda$ and the value $\frac{2(\|\vx\|_1-n\lambda)}{-\lambda}$ is the associated eigenvalue. Simplifying the above equation leads to the following quadratic equation
$$
\frac{\rho}{2}\|\vx\|_1 \lambda^2-\left(\frac{\rho}{2}\|\vx\|_2^2+n\right)\lambda+\|\vx\|_1=0.
$$
The discriminant of the quadratic equation with variable $\lambda$ is $\Delta$ given by \eqref{eq:Delta}. Since $\|\vx\|_1^2< n \|\vx\|_2^2$, we have
$\Delta>(\frac{\rho}{2}\|\vx\|_2^2-n)^2 \ge 0$.  Hence, the above quadratic equation has two real roots $\lambda_1=\frac{\overline{\alpha}}{\rho \|\vx\|_1}$ and $\lambda_2=\frac{\underline{\alpha}}{\rho \|\vx\|_1}$. Substituting $\lambda_1$ and $\lambda_2$ into $\frac{2(\|\vx\|_1-n\lambda)}{-\lambda}$ yield two eigenvalues $\overline{\lambda}$ and $\underline{\lambda}$ of $\mA_{\rho,\vx}$, respectively. In this case, we know that $\overline{\lambda}> 0$ and $\underline{\lambda}< 0$. The eigenvectors corresponding to $\overline{\lambda}$ and $\underline{\lambda}$ are $\overline{\vw}$ and $\underline{\vw}$, respectively.
\end{proof}

We remark that for $\vx \in \mathbb{R}^n_{\downarrow}$, the largest component of $\underline{\vw}$ in \eqref{eq:u-under}, that is the first component $x_1-\frac{\underline{\alpha}}{\rho \|\vx\|_1}$, is always non-negative. Actually, by \eqref{eq:Delta}, \eqref{eq:alpha-under}, and \eqref{eq:u-under}, we have
\begin{eqnarray*}
x_1-\frac{\underline{\alpha}}{\rho \|\vx\|_1}&=&x_1-\frac{2\rho \|\vx\|_1^2}{\rho \|\vx\|_1 \left[\left(\frac{\rho}{2}\|\vx\|_2^2+n\right)+\sqrt{\left(\frac{\rho}{2}\|\vx\|_2^2+n\right)^2-2\rho \|\vx\|_1^2}\right]} \\
&\ge& x_1-\frac{2\|\vx\|_1}{\left(\frac{\rho}{2}\|\vx\|_2^2+n\right)+\left|\frac{\rho}{2}\|\vx\|_2^2-n\right|} =\left\{
         \begin{array}{ll}
          x_1- \frac{2\|\vx\|_1}{\rho \|\vx\|_2^2}, & \hbox{if $\rho \|\vx\|_2^2 \ge 2n$;} \\
          x_1- \frac{2\|\vx\|_1}{2n}, & \hbox{if $\rho \|\vx\|_2^2 < 2n$,}
         \end{array}
       \right.
\end{eqnarray*}
which is always non-negative. This derivation also indicates that $x_1-\frac{\underline{\alpha}}{\rho \|\vx\|_1}>0$ always holds if $\vx$ is not parallel to $\ve$.

If the last component $x_n-\frac{\underline{\alpha}}{\rho \|\vx\|_1}$ of $\underline{\vw}$ in \eqref{eq:u-under} is positive, we have the following result.
\begin{theorem}\label{thm:positive-comp}
For $\rho>0$ and  $\vx \in \Rnd$ not being a multiple of $\ve$, if $x_n > \frac{\underline{\alpha}}{\rho \|\vx\|_1}$, then the vector  $\vw^\star:=\frac{\underline{\vw}}{\|\underline{\vw}\|_2}$ with $\underline{\vw}=\vx-\frac{\underline{\alpha}}{\rho \|\vx\|_1}\ve$ is the solution to the optimization problem \eqref{w-star}. Furthermore, we have
$$
\mathrm{prox}_{\frac{1}{\rho} h_2}(\vx)=\left\langle \vx, \vw^\star\right \rangle \vw^\star.
$$
\end{theorem}
\begin{proof}\ \ From  $\vx \in \Rnd$ not being a multiple of $\ve$, $x_n > \frac{\underline{\alpha}}{\rho \|\vx\|_1}$, and $\underline{\alpha}$ being nonnegative, we know that $\mathbf{0} \neq \underline{\vw} \in \Rnd$ and $\vw^\star \in \Rnd \cap \mathbb{S}^{n-1}_{+}$. By identifying $\mA_{\rho, \vx}$, $\mathbf{0}$, and $1$ as $\mH$, $\vb$, and $r$ in \eqref{eq:tao} of Lemma~\ref{lemma:Tao-An}, respectively, we know that $\mA_{\rho, \vx}-\underline{\lambda}\mI$ is  positive semi-definite and $(\mA_{\rho, \vx}-\underline{\lambda}\mI)\underline{\vw}=\mathbf{0}$ from the item (ii) of Proposition~\ref{prop:A}. Therefore, the unit vector $\vw^\star \in \R_{+}^n$ is the solution to  the problem \eqref{w-star} from Lemma~\ref{lemma:Tao-An}.

To determine $\mathrm{prox}_{\frac{1}{\rho} h_2}(\vx)$, we notice that the first entries of both $\vx$ and $\vw^\star$ are positive, hence $\langle \vx, \vw^\star\rangle>0$. Furthermore, since  $G(\vw^\star)=\underline{\lambda}<0$ for $G$ given in \eqref{def:G}, we conclude that $F(\vw^\star)<F(\mathbf{0})$ for $F$ given in \eqref{def:F}. This completes the proof of this theorem.
\end{proof}

There are two remarks on Theorem~\ref{thm:positive-comp}. The first one is that under the conditions of this theorem, simplifying the expression  $\left\langle \vx, \vw^\star\right \rangle \vw^\star$ leads to
$$
\mathrm{prox}_{\frac{1}{\rho} h_2}(\vx)=\frac{\|\vx\|_2^2-\frac{\underline{\alpha}}{\rho}}{\|\vx\|_2^2-2\frac{\underline{\alpha}}{\rho}+\frac{n \underline{\alpha}^2}{\rho^2 \|\vx\|_1^2}} \left(\vx-\frac{\underline{\alpha}}{\rho \|\vx\|_1}\ve\right).
$$
The second remark concerns the consistency of Theorem~\ref{thm:positive-comp} in $\mathbb{R}^2_{\downarrow}$ with Theorem~\ref{thm:n=2}. That is, if the condition $x_2 > \frac{\underline{\alpha}}{\rho \|\vx\|_1}$ holds, then $\rho x_1 x_2 >2$ and $\vw^\star$ in both Theorem~\ref{thm:positive-comp} and Theorem~\ref{thm:n=2} are identical. To this end, and to have simpler expressions, let us denote
$$
a:=\sqrt{\left(\frac{\rho}{2}\|\vx\|_2^2+2\right)^2-2\rho\|\vx\|_1^2} \quad \mbox{and} \quad b:=\left(\frac{\rho}{2}\|\vx\|_2^2+2\right) - \rho x_2\|\vx\|_1.
$$
By \eqref{eq:alpha-under}, the condition $x_2 > \frac{\underline{\alpha}}{\rho \|\vx\|_1}$  implies $a>b$. We claim that $a>|b|$. If this claim does not hold, then $b$ must be negative and $0<a \le |b|$. Squaring this inequality and simplifying it yield  $\rho x_1 x_2 \le 2$. In this situation, $b=\frac{\rho}{2}(x_1^2 - x_2^2) + 2 - \rho x_1 x_2>0$. This contradicts the negativeness of $b$. Hence,  $a>|b|$. Similarly, squaring this inequality and simplifying it leads to  $\rho x_1 x_2 > 2$.

Further, defining $\beta := 2\theta^\star=\arctan \left(\frac{-2(2-\rho x_1x_2)}{\rho(x_1^2-x_2^2)}\right)$ and with the help of the identity
$$
\frac{\cos \theta^\star}{\sin \theta^\star} =\sqrt{1+\frac{1}{\tan^2 \beta}} + \frac{1}{\tan \beta},
$$
we can show, after some simplifications, that the ratios of the entries of $\vw^\star$ in both Theorem~\ref{thm:positive-comp} and Theorem~\ref{thm:n=2} are the same:
$$
\frac{\cos \theta^\star}{\sin \theta^\star} = \frac{\rho x_1\|\vx\|_1-\underline{\alpha}}{\rho x_2\|\vx\|_1-\underline{\alpha}},
$$
which means that $\vw^\star$ in both Theorem~\ref{thm:positive-comp} and Theorem~\ref{thm:n=2} are identical.

The next result discusses the property of the solution from the $\vw$-step under the condition that the last component $x_n-\frac{\underline{\alpha}}{\rho \|\vx\|_1}$ of $\underline{\vw}$ in \eqref{eq:u-under} is non-positive.
\begin{theorem}
For $\rho>0$ and $\vx\in \Rnd$, let $\vw^\star$ be the optimal solution to the optimization problem \eqref{w-star}. If $x_n\le \frac{\ul\alpha}{\rho\|\vx\|_1}$, then $\left(\vw^\star\right)_n=0$.
\end{theorem}
\begin{proof}\ \ Suppose that all components of $\vw^\star$ are positive, Then $\vw^\star \in \Ss_+^{n-1} \subset \mathbb{S}^{n-1}$. So $\vw^\star$ is a local minimizer of
\begin{equation} \label{eq:tao2}
    \min\left\{\frac{1}{2}\vw^\top\mA_{\rho, \vx}\vw: \vw\in \Ss^{n-1}\right\}.
\end{equation}
As the zero vector is orthogonal any vector, it naturally follows that it is orthogonal to $\underline{\vw}$, the eigenvector of $\mA_{\rho,\vx}$ associated with the negative eigenvalue $\underline{\lambda}$. By Lemma \ref{lemma:Tao-An2}, there is no local-nonglobal minimum for \eqref{eq:tao2}. Hence  $\vw^\star$ is the global minimizer of problem~\eqref{eq:tao2}. As a result, $\vw^\star  =  \frac{\underline{\vw}}{\|\underline{\vw}\|_2}$,
whose last component is less than $0$ by the given condition $x_n\le \frac{\ul\alpha}{\rho\|\vx\|_1}$. This completes our proof.
\end{proof}

To have an efficient approach for computing the proximity operator of $h_2$, let us access the entries of the matrix $\mA_{\rho, \vx}$, which are
$$
\mA_{\rho, \vx}=\begin{bmatrix}
    2-\rho x_1^2 & 2-\rho x_1 x_2& \cdots &2-\rho x_1 x_n \\
    2-\rho x_2 x_1 & 2-\rho x_2^2& \cdots &2-\rho x_2 x_n \\
    \vdots&\vdots&\ddots&\vdots\\
    2-\rho x_n x_1 & 2-\rho x_n x_2& \cdots &2-\rho x_n^2
\end{bmatrix}.
$$
Since $\vx \in \mathbb{R}^{n}_{\downarrow}$, the numbers of entries in each row, each column, and each diagonal are increasing corresponding to the indices of the entries. Based on the structure of this matrix, we define a function $\mu$ that maps every pair $(\rho, \vx)$ with $\rho$ and $\vx \in \mathbb{R}^n_{\downarrow}$ to a non-negative integer as follows:
\begin{equation}\label{eq:mu}
\mu(\rho, \vx):=\left\{
                 \begin{array}{ll}
                   0, & \hbox{if $(\mA_{\rho,\vx})_{11} \ge 0$;} \\
                   k, & \hbox{if there exists $1\le k < n$ such that $(\mA_{\rho,\vx})_{1k} < 0$ and $(\mA_{\rho,\vx})_{1(k+1)} \ge  0$;} \\
                   n, & \hbox{if $(\mA_{\rho,\vx})_{1n}  < 0$.}
                 \end{array}
               \right.
\end{equation}
This number $\mu(\rho, \vx)$ counts how many negative components in the vector $2\ve-\rho x_1 \vx$. As $2\ve-\rho x_1 \vx$ is the first column of the matrix $\mA_{\rho,\vx}$, with the number $\mu(\rho, \vx)$, we consider three cases for the matrix $\mA_{\rho,\vx}$ in the following theorem.

\begin{theorem}\label{thm:Main}
Let $\rho>0$ and let $\vx \in \mathbb{R}_{\downarrow}^n$. Set $k=\mu(\rho, \vx)$. Then the following statements hold:
\begin{itemize}
\item[(i)] If $k=0$, then $\ve_1$ is the global minimizer to the optimization problem \eqref{w-star};
\item[(ii)] If $1\le k \le n$, then the vector
$$
\begin{bmatrix}\tilde{\vw}^\star \\ \mathbf{0}_{(n-k)\times 1}\end{bmatrix}
$$
is the global minimizer to the optimization problem \eqref{w-star}, where $\tilde{\vw}^\star$ is the minimizer of the problem
$$
\min_{\vw \in \mathbb{S}^{k-1}_{+}}\frac{1}{2} \vw^\top \mA_{\rho,\vx_{[k]}} \vw.
$$
Here $\mA_{\rho,\vx_{[k]}}$ is the $k$-order leading principal submatrix of $\mA_{\rho,\vx}$ obtained by removing its last $(n-k)$ rows and columns.
\end{itemize}
\end{theorem}
\begin{proof}\ \
(i) For $\vx \in \mathbb{R}_{\downarrow}^n$, from the fact $(\mA_{\rho,\vx})_{11} \ge 0$, we conclude that $(\mA_{\rho,\vx})_{ij}\ge (\mA_{\rho,\vx})_{11} \ge 0$ for all $i,j \in [n]$.  Therefore, for all $\vw \in \mathbb{S}^{n-1}_{+}$, we have
$$
\frac{1}{2} \vw^\top \mA_{\rho,\vx} \vw
  \ge  \frac{1}{2} (\mA_{\rho,\vx})_{11}\sum_{i,j=1}^n w_iw_j =\frac{1}{2} (\mA_{\rho,\vx})_{11} \|\vw\|_1^2 \ge \frac{1}{2} (\mA_{\rho,\vx})_{11} \|\vw\|_2^2= \frac{1}{2} (\mA_{\rho,\vx})_{11}.
$$
The inequalities in the above can be achieved for $\vw=\ve_1$.


(ii) In this case, we split the matrix $\mA_{\rho,\vx}$ into $2 \times 2$ block matrix as follows
$$
\mA_{\rho,\vx} = \begin{bmatrix}\mA_{11}&\mA_{12} \\ \mA_{21}&\mA_{22} \end{bmatrix},
$$
where $\mA_{11}$, $\mA_{12}$, $\mA_{21}$, and $\mA_{22}$ are size $k \times k$, $k \times (n-k)$, $(n-k) \times k$, and $(n-k)\times (n-k)$, respectively. In fact, $\mA_{11}=\mA_{\rho,\vx_{[k]}}$. We further know that all entries in $\mA_{12}$, $\mA_{21}$, and $\mA_{22}$ are non-negative. For any $\vw \in \mathbb{S}^{n-1}_{+}$, write
$$
\vw=\begin{bmatrix}\vw_1 \\ \vw_2\end{bmatrix}
$$
with $\vw_1 \in \mathbb{R}^k$ and $\vw_2 \in \mathbb{R}^{n-k}$. We have
$$
\vw^\top \mA_{\rho,\vx} \vw=\vw_1^\top \mA_{11}\vw_1+\vw_1^\top \mA_{12}\vw_2+\vw_2^\top \mA_{21}\vw_1+ \vw_2^\top \mA_{22} \vw_2 \ge \vw_1^\top \mA_{11}\vw_1.
$$
The inequality $2-\rho x_1 x_k<0$ implies $\min_{\vw_1} \vw_1^\top \mA_{11}\vw_1 <0$. Thus,
$$
\min_{\vw \in \mathbb{S}^{n-1}_{+}}\frac{1}{2} \vw^\top \mA_{\rho,\vx} \vw \ge \min_{\vw \in \mathbb{S}^{n-1}_{+}}\frac{1}{2} \vw_1^\top \mA_{11} \vw_1 \ge \min_{\tilde{\vw} \in \mathbb{S}^{k-1}_{+}}\frac{1}{2} \tilde{\vw}^\top \mA_{11} \tilde{\vw}.
$$
In particular, for all vectors $\tilde{\vw} \in \mathbb{S}^{n-1}_{+}$ with $\vw_2=\mathbf{0}$, one has
$$
\frac{1}{2} \tilde{\vw_1}^\top \mA_{11} \tilde{\vw_1}=\frac{1}{2} \tilde{\vw}^\top \mA_{\rho,\vx} \tilde{\vw} \ge \min_{\vw \in \mathbb{S}^{n-1}_{+}}\frac{1}{2} \vw^\top \mA_{\rho,\vx} \vw.
$$
We conclude that
$$
\min_{\vw \in \mathbb{S}^{n-1}_{+}}\frac{1}{2} \vw^\top \mA_{\rho,\vx} \vw=\min_{\tilde{\vw} \in \mathbb{S}^{k-1}_{+}}\frac{1}{2} \tilde{\vw}^\top \mA_{11} \tilde{\vw}.
$$
This completes the proof.
\end{proof}

We remark that not all entries of $\tilde{\vw}^\star$ in Theorem~\ref{thm:Main} are necessarily positive, and some entries may be zero, as demonstrated in the following example.

\begin{example}
Let
$$
\vx = \begin{bmatrix}2.5&1.5&1&0.5\end{bmatrix}^\top.
$$
For this vector and two different values of $\rho$, we present the matrix $\mA_{\rho, \vx}$, its eigenvector $\vv$ associated with the negative eigenvalue, and $\vw^\star$ the minimizer of the problem $\min_{\vw \in \mathbb{S}^3_{+}}\frac{1}{2} \vw^\top \mA_{\rho, \vx} \vw$.

For $\rho_1=2.5$, we have $\mA_{\rho_1, \vx}$, $\vv_1$, and $\vw_1^\star$ as follows:
$$\mA_{\rho_1,\vx}=\frac{1}{8}\begin{bmatrix}
-109&   -59&   -34&    -9\\
   -59&   -29&   -14&     1\\
   -34&   -14&    -4&     6\\
    -9&     1&     6&    11
\end{bmatrix},
\vv_1=\begin{bmatrix}0.8598\\ 0.4481 \\ 0.2422 \\ 0.0363\end{bmatrix},
\quad \mbox{and} \quad
\vw_1^\star=\begin{bmatrix}0.8598\\ 0.4481 \\ 0.2422 \\ 0.0363\end{bmatrix}.
$$

For $\rho_2=1.8$, we have $\mA_{\rho_2,\vx}$, $\vv_2$, and $\vw_2^\star$ as follows:
$$
\mA_{\rho_2,\vx}=\begin{bmatrix}
   -9.25&   -4.75&   -2.50&   -0.25\\
   -4.75&   -2.05&   -0.70&    0.65\\
   -2.50&   -0.70&    0.20&    1.10\\
   -0.25&    0.65&    1.10&    1.55
\end{bmatrix},
\vv_2=\begin{bmatrix}0.8795\\    0.4294\\    0.2043\\   -0.0207\end{bmatrix},
\quad \mbox{and} \quad
\vw_2^\star=\begin{bmatrix}0.8804\\    0.4286\\    0.2027\\     0\end{bmatrix}.
$$

Notice that for the values $\rho_1=2.5$ and $\rho_2 = 1.8$, both meet the condition $2-\rho x_1x_4 <0$, that is $\mu(\rho_1,\vx)=\mu(\rho_2,\vx)=4$. However, this does not determine the positivity of all components in $\vw^\star$.
\end{example}

We can establish that $h_2$ acts as a promoter of sparsity from Theorem~\ref{thm:Main} under the situation of $\mu(\rho,\vx)=0$. This assertion is encapsulated in the subsequent result.

\begin{theorem}
For $\rho>0$, the following inclusion holds for all $\vx$ in the set $\{\vx \in \mathbb{R}^n: \|\vx\|_\infty \le \sqrt{{2}/{\rho}}\}$:
$$
\mathbf{0} \in \mathrm{prox}_{\frac{1}{\rho}h_2}(\vx).
$$
\end{theorem}
\begin{proof} By Lemma~\ref{lem:properties},
it suffices to consider all points in  the set $\mathbb{R}^n_\downarrow$ with their $\ell_\infty$ norm smaller than $\sqrt{{2}/{\rho}}$. For $\vx \in \mathbb{R}^n_\downarrow$, we examine two scenarios.
If $\vx=\alpha \ve$ with $\alpha \le \sqrt{{2}/{\rho}}$, the result holds due to  Theorem~\ref{thm:positive-multiple-1}. If $\vx \neq\alpha \ve$ for any $\alpha>0$, by  Theorem~\ref{thm:Main} we have $G(\ve_1)=\frac{1}{2}(2-\rho x_1^2)\ge 0$, hence, the results holds as well.
\end{proof}
This theorem underscores the sparse-promoting nature of $h_2$ within the specified domain.

Given $\rho > 0$ and $\vx \in \Rnd$, Theorem~\ref{thm:Main} provides a clear guideline for algorithm development when computing the optimal solution $\vw$ to problem~\eqref{w-star}, eventually, $\mathrm{prox}_{\frac{1}{\rho}h_2}(\vx)$. If there exists an integer $k \in [1, n-1]$ such that $2 - \rho x_1x_k < 0$ and $2 - \rho x_1x_{k+1} \geq 0$, it follows that $w_{k+1} = \cdots = w_n = 0$. This allows us to safely truncate $\vx$ by removing its last $n-k$ entries. This approach can significantly speed up the computation process by focusing only on the relevant components of $\vx$.

We are ready now to present our algorithm for computing  $\prox_{\frac{1}{\rho} h_2}$ based on our WRD procedure for arbitrary $\vx \in \mathbb{R}^n$.  This algorithm is presented in Algorithm~\ref{alg:prox}.
\begin{algorithm}[h]
\caption{Computing the Proximal Operator of $h_2$} \label{alg:prox}
\begin{algorithmic}[1]
\State \textbf{Input:} Vector $\vx \in \mathbb{R}^n$, parameter $\rho > 0$
\State \textbf{Output:} The proximal operator $\text{prox}_{\frac{1}{\rho} h_2}(\vx)$
\Procedure{}{WRD Procedure}
    \State Sort and convert $\vx$ into $\mathbb{R}^n_{\downarrow}$ via a signed permutation matrix $\mP$.
    \State Compute $k=\mu(\rho, \vx)$ by \eqref{eq:mu}
    \If {$k=0$}
        \State {$\vw=\ve_1$ \hskip 1cm (see item (i) of Theorem~\ref{thm:Main}) }
    \Else {\hfill \texttt{($\vw$-step)}}
        \For{ $k:-1:1$}
            \State{Forming a vector (still denoted by $\vx$) from the first $k$ entries of $\vx$}
            \If {$\vx=\alpha \ve$ for some $\alpha>0$}
                \State $\vu=\text{prox}_{\frac{1}{\rho} h_2}(\vx)$ by Theorem~\ref{thm:positive-multiple-1}
            \ElsIf {$k=2$}
                \State {return $\vw$ by Theorem~\ref{thm:n=2}}
            \Else
            \If{the last entry of $\underline{\vw}$ by \eqref{eq:u-under}, is greater than $0$}
                \State {return $\vw \gets \frac{\underline{\vw}}{\|\underline{\vw}\|_2}$ by Theorem~\ref{thm:positive-comp}}
            \EndIf
            \EndIf
        \EndFor
    \EndIf
    \State {Pad $\vw$ with a zero block such that the resulting vector, still denoted by $\vw$, is in $\mathbb{S}^{n-1}_{+}$.}
    \State {Form $\vu \gets \langle \vx, \vw\rangle \vw$ \hfill  \texttt{($r$-step)}}
    \State {Determine $\vu \gets \left\{
                 \begin{array}{ll}
                   \mathbf{0}, & \hbox{if $F(\mathbf{0})\le F(\vu)$;} \\
                   \vu, & \hbox{otherwise.}
                 \end{array}
               \right.$ \hfill \texttt{($d$-step)}}
    \State{$\vu \gets \mP^{-1}\vu \in \text{prox}_{\frac{1}{\rho} h_2}(\vx)$}
\EndProcedure
\end{algorithmic}
\end{algorithm}

\section{The Proximal Operator of $h_1$} \label{sec:h1}
In this section, we detail the computation of the proximal operator for the function $h_1$ via the WRD procedure.

We begin with showing the optimization problem \eqref{w-star} associated with the $\vw$-step of the WRD procedure. For the given $\rho>0$ and $\vx\in \R_+^n$,
defining
\begin{equation}  \label{def:A1}
    \mA_{\rho, \vx}= -\rho \cdot \vx\vx^{\top}.
\end{equation}
The corresponding function $G$ in \eqref{def:G} for $h_1$ becomes the quadratic form
$$
G(\vw)=\frac{1}{2}\vw^{\top}\mA_{\rho, \vx} \vw+\ve^{\top} \vw.
$$
By Lemma~\ref{lem:properties},  our focus is restricted to discussing the proximity operator of $h_1$ on $\Rnd$. This discussion unfolds in the subsequent three subsections.

In the first subsection, we highlight that the method for $h_2$, as delineated in Section~\ref{sec:h2}, cannot be directly applied to $h_1$, despite the initial feasibility of such a transfer, particularly considering their analogous reformulations. Additionally, we provide the explicit expression of the proximity operator of $h_1$
at specific points, highlighting that $h_1$ serves as a function that promotes sparsity.

The second subsection conducts an in-depth examination of the proximity operator of $h_1$ in $\mathbb{R}^2$. Notably, the method tailored for this task poses challenges in its extension to higher dimensions.

In the third subsection, we introduce a strategy to transform the optimization problem in the $\vw$-step of the WRD procedure. This transformation entails converting a concave objective function constrained on a nonconvex set into one with the same objective function but constrained on a closed and bounded convex set. The latter can be efficiently solved using the nonconvex gradient projection algorithm (see \cite{Attouch-Bolte-Svaiter:MP:13}).

\subsection{The approach for $h_2$ does not work for $h_1$}

Initially, it may seem feasible to directly apply the method for $h_2$ described in Section~\ref{sec:h2} to $h_1$, especially given their similar reformulations. However, we want to point out that this approach is not directly transferable to $h_1$. This becomes evident when considering Lemma~\ref{lemma:Tao-An}, which leads us to the subsequent result.

\begin{proposition}\label{prop:optimal}
For $\vx \in \R_+^n$  and $\rho>0$, we consider a quadratic optimization problem on the unit sphere as follows
\begin{equation} \label{model:sphere}
    \min\left\{\frac{1}{2}\vw^{\top}\mA_{\rho, \vx} \vw+\ve^{\top} \vw: \vw\in \mathbb{S}^{n-1}\right\}.
\end{equation}
A vector $\vw^\star$ is a solution to \eqref{prop:optimal} if and only if there is a unique $\lambda^\star>\rho \|\vx\|_2^2$ such that
$$
(\mA_{\rho,\vx}+\lambda^\star \mI)\vw^\star=-\ve
$$
with $\vw^\star$ being a unit vector.
\end{proposition}
\begin{proof} Problem~\eqref{model:sphere} is a special case of problem~\eqref{eq:tao} by identifying $\mA_{\rho, \vx}$, $\ve$, and $1$ as $\mH$, $\vb$, and $r$, respectively.

The matrix $\mA_{\rho, \vx}=-\rho \vx \vx^\top$ is a rank-1 matrix and has $-\rho \|\vx\|_2^2$ as its only one non-zero eigenvalue with the associated unit eigenvector $\frac{\vx}{\|\vx\|_2}$. Hence, for any $\lambda \ge \rho \|\vx\|_2^2$, the matrix $\mA_{\rho,\vx}+\lambda \mI$ is positive semidefinite.

``$\Rightarrow$'' If $\vw^*$ is the optimal solution to problem~\eqref{model:sphere}, by Lemma~\ref{lemma:Tao-An}, there exists a unique $\lambda^\star \ge \rho \|\vx\|_2^2$ such that $(\mA_{\rho,\vx}+\lambda^\star \mI)\vw^\star=-\ve$ with $\vw^\star$ being a unit vector. We claim that $\lambda^\star > \rho \|\vx\|_2^2$. If not, assume that $\lambda^\star = \rho \|\vx\|_2^2$, and let $\mU$ be an orthogonal matrix whose the first column is $\frac{\vx}{\|\vx\|_2}$. Then, the equality $(\mA_{\rho,\vx}+\lambda^\star \mI)\vw^\star=-\ve$ leads to
$$
\mU \begin{bmatrix}
    0&&&\\
    &\rho \|\vx\|_2^2&&\\
    &&\ddots&\\
    &&&\rho \|\vx\|_2^2
\end{bmatrix} \mU^\top \vw^\star=-\ve
\quad \mbox{or} \quad
\begin{bmatrix}
    0&&&\\
    &\rho \|\vx\|_2^2&&\\
    &&\ddots&\\
    &&&\rho \|\vx\|_2^2
\end{bmatrix} \mU^\top \vw^\star=-\begin{bmatrix}
    \frac{\|\vx\|_1}{\|\vx\|_2} \\ \star \\ \vdots \\ \star
\end{bmatrix},
$$
which is inconsistent. Hence, $\lambda^\star$ is strictly greater than $\rho \|\vx\|_2^2$.

``$\Leftarrow$''  We show that there exists an $\lambda>\rho \|\vx\|_2^2$ such that $\|(\mA_{\rho,\vx}+\lambda \mI)^{-1}\ve\|_2=1$. For $\lambda \neq \rho \|\vx\|_2^2$, the matrix $\mA_{\rho,\vx}+\lambda \mI$ is invertible and its inverse is
$$
(\mA_{\rho,\vx}+\lambda \mI)^{-1}=\frac{1}{\lambda}\left(\mI+\frac{\rho}{\lambda-\rho \|\vx\|_2^2}\vx \vx^\top\right).
$$
For $\lambda>\rho \|\vx\|_2^2$, from $\|(\mA_{\rho,\vx}+\lambda \mI)^{-1}\ve\|_2=1$ together with the above equation, we obtain
\begin{equation}\label{eq:Opt-General}
\left\|(\lambda-\rho \|\vx\|_2^2)\ve+\rho \|\vx\|_1\vx\right\|_2=\lambda (\lambda-\rho \|\vx\|_2^2).
\end{equation}
To study the root the above equation, we consider two different cases: (i) $\vx=\alpha \ve$ for some $\alpha>0$ and (ii) $\vx \neq \alpha \ve$ for any $\alpha>0$.

For the case of $\vx=\alpha \ve$ for some $\alpha>0$, one has $\|\vx\|_1=\alpha n$ and $\|\vx\|_2=\alpha \sqrt{n}$. It leads from \eqref{eq:Opt-General} that
$\lambda\sqrt{n}=\lambda(\lambda-\rho\alpha^2 n)$.  This equation has two real roots and the only root, that is larger than $\rho \|\vx\|_2^2$, is
$\lambda^\star = \sqrt{n}+\rho \alpha^2 n > \rho \alpha^2 n=\rho \|\vx\|_2^2$.
By Lemma~\ref{lemma:Tao-An},
\begin{equation}\label{eq:Case1}
\vw^\star = -\frac{1}{\sqrt{n}} \ve
\end{equation}
is the optimal solution to problem~\eqref{model:sphere}.

The rest of the proof considers the case of $\vx \neq \alpha \ve$ for any $\alpha>0$. Squaring the identity \eqref{eq:Opt-General} from its both sides and simplifying the resulting equation lead to the following quartic equation
$$
Q(q)=0,
$$
where $q=\lambda-\rho \|\vx\|_2^2$ and
$$
Q(q)=q^4+2\rho\|\vx\|_2^2 q^3+(\rho^2\|\vx\|_2^4-n)q^2-2\rho\|\vx\|_1^2q-\rho^2\|\vx\|_1^2\|\vx\|_2^2.
$$
Since $Q(0)=-\rho^2\|\vx\|_1^2\|\vx\|_2^2<0$ and $Q(q)$ is positive for a sufficient large $q$, there exists at least one root of $Q$ on the interval $[0,\infty)$. No matter what value of $(\rho^2\|\vx\|_2^4-n)$ will be, the number of sign changes of the polynomial $Q$ is $1$. Therefore, by Descartes' Rule of Signs \cite{Wang:AMM:2004}, we conclude that $Q$ has exactly one positive root, say $q^\star$.  Hence, with $\lambda^\star = q^\star+\rho \|\vx\|_2^2$,
\begin{equation}\label{eq:Case2}
\vw^\star= (\mA_{\rho, \vx}+\lambda^* \mI)^{-1} (-\ve) =-\frac{1}{\lambda^*}\left(\ve+\frac{\rho\|\vx\|_1}{\lambda^*-\rho\|\vx\|_2^2}\vx\right)
\end{equation}
is the optimal solution to problem~\eqref{model:sphere} by Lemma~\ref{lemma:Tao-An} again.
\end{proof}

It is evident from the preceding proof that all entries of the optimal solution $\vw^\star$, as indicated in \eqref{eq:Case1} and \eqref{eq:Case2}, are negative. Consequently, this vector $\vw^\star$ cannot serve as the solution to problem ~\eqref{w-star}. Therefore, the methodology employed for $h_2$ is not applicable to $h_1$, necessitating a distinct approach.

Next, we provide the proximity operator of $h_1$ for vectors $\vx$ with uniform entries.

\begin{theorem}\label{thm:h1-R2-ae}
For $\rho > 0$ and $\vx = \alpha \ve \in \mathbb{R}^n$ for some $\alpha > 0$, then
\[
\mathrm{prox}_{\frac{1}{\rho} h_1}(\vx) =
\begin{cases}
 \{\vzero\},&\text{if } \alpha<\sqrt{\frac{2}{\rho\sqrt{n}}} \\
\{\vzero, \vx\}, & \text{if } \alpha=\sqrt{\frac{2}{\rho\sqrt{n}}}; \\
   \{\vx\}, & \text{if } \alpha>\sqrt{\frac{2}{\rho\sqrt{n}}}.
\end{cases}
\]
\end{theorem}
\begin{proof}
In this situation, we have $\mA_{\rho, x} = - \rho\alpha^2 \ve\ve^{\top}$ from \eqref{def:A1}. The objective function of problem \eqref{w-star} is
\[
G(\vw)=\frac{1}{2} \vw^\top \mA_{\rho, x} \vw +\ve^{\top} \vw = -\frac{1}{2} \rho\alpha^2 \| \vw \|_1^2+ \|\vw\|_1=-\frac{1}{2} \rho\alpha^2\left(\|\vw\|_1-\frac{1}{\rho \alpha^2}\right)^2+\frac{1}{2\rho \alpha^2},
\]
where $\vw \in \mathbb{S}^{n-1}_{+}$.
Note that $\|\vw\|_1 \in [1, \sqrt{n}]$ for all $\vw \in \mathbb{S}^{n-1}_{+}$, the above quantity achieves its global minimum at $\|\vw\|_1$ being $1$ or $\sqrt{n}$, depending on which one is further away to $\frac{1}{\rho \alpha^2}$. Hence, $\|\vw^\star\|_1$ the $\ell_1$ norm of the optimal solution $\vw^\star$ to  problem \eqref{w-star} is $\sqrt{n}$ if $\frac{1}{\rho \alpha^2}<\frac{1}{2}(1+\sqrt{n})$; $1$ or $\sqrt{n}$ if $\frac{1}{\rho \alpha^2}=\frac{1}{2}(1+\sqrt{n})$; or $1$ if $\frac{1}{\rho \alpha^2}>\frac{1}{2}(1+\sqrt{n})$.
As a result, the $\vw$-step of the WRD procedure provides the optimal solution $\vw^\star$ to  problem \eqref{w-star} as follows:
$$
\vw^\star\in\begin{cases}
    \{\frac{1}{\sqrt{n}}\ve\}, & \text{if } \frac{1}{\rho \alpha^2}<\frac{1}{2}(1+\sqrt{n}); \\
\{\frac{1}{\sqrt{n}}\ve\}\cup \{\vect{e_i}: i=1,\ldots, n\}, & \text{if } \frac{1}{\rho \alpha^2}=\frac{1}{2}(1+\sqrt{n}); \\
\{\vect{e_i}: i=1,\ldots, n\},&\text{if } \frac{1}{\rho \alpha^2}>\frac{1}{2}(1+\sqrt{n}).
\end{cases}
$$

The $r$-step of the WRD procedure simply follows with $r^\star = \la \vx, \vw^\star \ra$. At the $d$-step of the WRD procedure, we compare $F(r^\star \vw^\star)$ and $F(\vzero)$ with $F$ defined in \eqref{def:F}. Note that
$$
F(r^\star \vw^\star)-F(\vzero)=G(\vw^\star)=
\begin{cases}
    -\frac{1}{2}\rho \alpha^2 n+\sqrt{n}, & \text{if } \frac{1}{\rho \alpha^2}<\frac{1}{2}(1+\sqrt{n}); \\
\frac{\sqrt{n}}{1+\sqrt{n}}, & \text{if } \frac{1}{\rho \alpha^2}=\frac{1}{2}(1+\sqrt{n}); \\
-\frac{1}{2}\rho \alpha^2 +1,&\text{if } \frac{1}{\rho \alpha^2}>\frac{1}{2}(1+\sqrt{n}).
\end{cases}
$$
We see that under the condition $\frac{1}{\rho \alpha^2}<\frac{1}{2}(1+\sqrt{n})$, the quality $F(r^\star \vw^\star)-F(\vzero)=-\frac{1}{2}\rho \alpha^2 n+\sqrt{n}$ is positive if $\frac{1}{\rho \alpha^2}>\frac{\sqrt{n}}{2}$, zero if $\frac{1}{\rho \alpha^2}=\frac{\sqrt{n}}{2}$, or negative if $\frac{1}{\rho \alpha^2}<\frac{\sqrt{n}}{2}$;  Under the condition $\frac{1}{\rho \alpha^2}=\frac{1}{2}(1+\sqrt{n})$, $F(r^\star \vw^\star)-F(\vzero)=\frac{\sqrt{n}}{1+\sqrt{n}}>0$; Under the condition $\frac{1}{\rho \alpha^2}>\frac{1}{2}(1+\sqrt{n})$, i.e., $-\frac{1}{2}\rho \alpha^2>\frac{-1}{1+\sqrt{n}}$, we have $F(r^\star \vw^\star)-F(\vzero)=-\frac{1}{2}\rho \alpha^2 +1>\frac{\sqrt{n}}{1+\sqrt{n}}$ always positive. The result of this theorem follows from \eqref{x-star}.
\end{proof}

The next result shows that the function $h_1$ is indeed a sparse promoting function whose proximity operator will send the points in a neighborhood of the origin to the origin (see \cite{Shen-Suter-Tripp:JOTA:2019}).
\begin{theorem}\label{thm:h1-sparse}
For $\rho > 0$, the following inclusion
$$
\mathbf{0} \in \mathrm{prox}_{\frac{1}{\rho} h_1}(\vx)
$$
holds for $\vx \in \mathbb{R}^n_{+}$ with $\|\vx\|_2 \le \sqrt{\frac{2}{\rho}}$.
\end{theorem}
\begin{proof}\ \ Let $G$ be the objective function of problem \eqref{w-star} associated with $h_1$. For $\vw \in \mathbb{S}^{n-1}_{+}$, we have
\begin{equation*}
G(\vw)=-\frac{\rho}{2} \langle \vx, \vw\rangle^2 +\ve^\top \vw
\ge -\frac{\rho}{2} \|\vx\|_2^2 + 1 \ge 0
\end{equation*}
for $\vx \in \mathbb{R}^n_{+}$ with $\|\vx\|_2 \le \sqrt{\frac{2}{\rho}}$.   We further have
$F(\langle \vx, \vw\rangle \vw) -F(\mathbf{0})=G(\vw)\ge 0$,
where $F$ is defined in \eqref{def:F}. Hence, $\mathbf{0} \in \mathrm{prox}_{\frac{1}{\rho} h_1}(\vx)$.
\end{proof}

\subsection{Special case: the proximity operator of $h_1$ on $\R^2$}

The following result establishes a region in which the proximity operator of $h_1$ does not vanish on $\R_{\downarrow}^2$.
\begin{proposition}\label{prop:region-crude-h1}
For $\rho>0$, define two sets in $\R_{\downarrow}^2$ as follows:
\begin{eqnarray*}
    S_1&=&\left\{\vx \in \R_{\downarrow}^2: x_1 > \sqrt{\frac{2}{\rho}}\right\},\\
    S_2&=&\left\{\vx \in \R_{\downarrow}^2: x_2=\kappa x_1,  x_1>\sqrt{\frac{2(1+\kappa)}{\rho(1+\kappa^2)^{3/2}}}, \kappa \in [0,1]\right\}.
\end{eqnarray*}
Then, the origin is not in $\mathrm{prox}_{\frac{1}{\rho}h_1}(\vx)$ for every point $\vx \in S_1 \cup S_2$.
\end{proposition}
\begin{proof} For each point $\vx \in S_1 \cup S_2$, to prove the origin is not in $\mathrm{prox}_{\frac{1}{\rho}h_1}(\vx)$  it is sufficient to show that there exists a point, say $\vz$, in $\R_{\downarrow}^2$ such that $F(\vz)-F(\mathbf{0})<0$, where $F$ is defined in \eqref{def:F}.

First, we choose $\vz=x_1\ve_1\in \R_{\downarrow}^2$. Then, $F(\vz)-F(\mathbf{0})=-\frac{\rho}{2}x_1^2+1<0$ which holds for $\vx \in S_1$.

Next, we choose $\vz=\vx$. Then, with $\kappa=\frac{x_2}{x_1}$,
$$
F(\vz)-F(\mathbf{0})=(1+\kappa^2)\left(-\frac{1}{2}\rho x_1^2 + \frac{1+\kappa}{(1+\kappa^2)^{3/2}}\right)<0,
$$
for all points $\vx \in S_2$. This completes the proof of this proposition.
\end{proof}

We comment on this proposition. Consider two curves parameterized by the parameter $\kappa \in [0,1]$ as follows:
$$
\mathcal{C}_1: [0,1] \ni \kappa \mapsto \sqrt{\frac{2}{\rho}}(1,\kappa) \quad \mbox{and} \quad
\mathcal{C}_2: [0,1]\ni \kappa \mapsto \sqrt{\frac{2(1+\kappa)}{\rho(1+\kappa^2)^{3/2}}}(1,\kappa).
$$
We have $\mathcal{C}_1(0)=\mathcal{C}_2(0)=\sqrt{\frac{2}{\rho}}(1,0)$, $\mathcal{C}_1(1)=\sqrt{\frac{2}{\rho}}(1,1)$, and $\mathcal{C}_2(1)=\sqrt{\frac{\sqrt{2}}{\rho}}(1,1)$. Two curves intersect at the point with $\kappa$ to be the root of the polynomial of $\kappa^5+3\kappa^2+2\kappa-2=0$. This root is $\kappa\approx 0.6124$. The red shaded region in Figure~\ref{fig:Prox_h1_Sparse_Region} is the set $S_1\cup S_2$.  The blue shaded region Figure~\ref{fig:Prox_h1_Sparse_Region} represents the set where every point is mapped to the origin by $\mathrm{prox}_{\frac{1}{\rho}h_1}(\vx)$, as stipulated by  Theorem~\ref{thm:h1-sparse}. We will explore the blank region situated between the blue and red shaded areas in the subsequent analysis.

\begin{figure}[ht]
    \centering
{
    \begin{tabular}{c}
    \includegraphics[width=4cm]{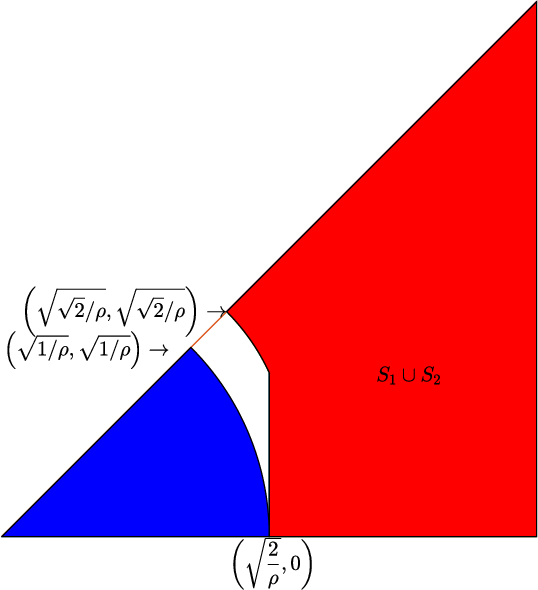}
    \end{tabular}
}
    \caption{The proximity operator $\mathrm{prox}_{\frac{1}{\rho} h_1}$ will map all points in the blue shaded region to the origin and all points in the red region to a nonzero point.}
    \label{fig:Prox_h1_Sparse_Region}
\end{figure}

In the following analysis, our discussion distinctly excludes the instances of uniform entries in $\vx$, which have been previously addressed in Theorem~\ref{thm:h1-R2-ae}. We now focus on the case where $\vx \in \R_{\downarrow}^2$. This scenario can be further divided into two distinct cases: one where $\vx$ contains one zero entry, and another where it does not. We begin by examining the situation where $\vx$ includes one zero entry, as detailed in the following proposition.

\begin{proposition}\label{prop:h1:n=2-x2=0}
For $\rho>0$ and $\vx=\alpha \ve_1$ with $\alpha>0$, then
$$
\mathrm{prox}_{\frac{1}{\rho}h_1}(\vx) = \begin{cases}
    \{\mathbf{0}\}, & \text{if } \alpha < \sqrt{\frac{2}{\rho}}; \\
    \{\mathbf{0}, \vx\}, & \text{if } \alpha = \sqrt{\frac{2}{\rho}}; \\
    \{\vx\}, & \text{if } \alpha > \sqrt{\frac{2}{\rho}}.
\end{cases}
$$
\end{proposition}
\begin{proof} \ \
The objective function of problem \eqref{w-star} $G$ associated with $h_1$ for the given $\vx$ is
$$
G(\vw)=-\frac{1}{2}\rho \alpha^2 w_1^2 + w_1 +w_2=-\frac{1}{2}\rho \alpha^2 w_1^2 + w_1 +\sqrt{1-w_1^2},
$$
where $w_1 \in [0,1]$. A direct calculation shows that both  functions $-\frac{1}{2}\rho \alpha^2 w_1^2$ and   $w_1 +\sqrt{1-w_1^2}$ are concave with respect to $w_1$. Together with the facts of $G(\ve_1)=-\frac{1}{2}\rho \alpha^2  + 1$ and $G(\ve_2)=1$, hence, $G$ achieves its global minimum at $\vw^\star=\ve_1$.

The $r$-step of the WRD procedure simply follows with $r^\star = \la \vx, \vw^\star \ra=\alpha$. At the $d$-step, we compare $F(r^\star \vw^\star)$ and $F(\vzero)$ via their difference $F(r^\star \vw^\star)-F(\vzero)=G(\vw^\star)=-\frac{1}{2}\rho \alpha^2  + 1$.
Our result of this theorem immediately follows from the above difference.
\end{proof}

We observe that Proposition~\ref{prop:h1:n=2-x2=0} corroborates the findings of Proposition~\ref{prop:region-crude-h1} for points lying on the $x_1$-axis. Further, $(\mathrm{prox}_{\frac{1}{\rho}h_1}(\vx))_1=\mathrm{prox}_{\frac{1}{\rho}|\cdot|_0}(x_1)$ for $\vx=\alpha \ve_1$.


For $\vx\in \R_{\downarrow}^2$ with $x_1 \neq 0$, let $G$ be the objective function of problem~\eqref{w-star} associated with $h_1$. We define $Q: [0, \frac{\pi}{4}] \rightarrow \R$ as
\begin{equation*}
Q(\theta):=G(\vw(\theta)) \quad \mbox{with} \quad \vw(\theta) =  \begin{bmatrix} \cos(\theta) \\ \sin(\theta)  \end{bmatrix}.
\end{equation*}
A direct computation yields
\begin{equation}\label{def:Q-h1-R2}
Q(\theta)=-\frac{1}{2}\rho \|\vx\|_2^2\cos^2\left(\theta-\frac{\alpha}{2}\right)+\sqrt{2}\sin\left(\theta+\frac{\pi}{4}\right),
\end{equation}
where the constant $\alpha$ is given by, with  $\kappa=\frac{x_2}{x_1} \in [0,1]$,
\begin{equation}\label{eq:alpha}
    \alpha = \left\{
\begin{array}{ll}
    \arctan \left(\frac{2\kappa}{1-\kappa^2}\right) \in \left[0, \frac{\pi}{2}\right), & \hbox{if $x_1>x_2$;} \\
    \frac{\pi}{2}, & \hbox{if $x_1=x_2$.}
  \end{array}
    \right.
\end{equation}
Then, solving problem~\eqref{w-star} involves  minimizing $Q$ over the interval  $[0, \frac{\pi}{4}]$.  The minimal value of $Q$ on this interval can be attained at $0$, $\pi/4$, or the critical points of $Q$. To determine these critical points, we examine the properties of  $Q'$, which is
$$
Q'(\theta)=\frac{1}{2}\rho \|\vx\|_2^2\sin(2\theta-\alpha)+\sqrt{2}\cos\left(\theta+\frac{\pi}{4}\right).
$$
We immediately observed that: first, the function $\sqrt{2}\cos(\theta+\frac{\pi}{4})$  monotonically decreases from $1$ to $0$ as $\theta$ varies from $0$ to $\frac{\pi}{4}$; second, the function $\frac{1}{2}\rho \|\vx\|_2^2\sin(2\theta+\alpha)$ monotonically increases from $\frac{1}{2}\rho\|\vx\|_2^2 \sin(-\alpha)=-\rho x_1 x_2$ to $0$ as $\theta$ ranges from $0$ to $\frac{\alpha}{2}$, and from $0$ to $\frac{1}{2}\rho \|\vx\|_2^2 \cos(\alpha)=\frac{1}{2}\rho(x_1^2 - x_2^2)$ as $\theta$ goes from $\frac{\alpha}{2}$ to $\frac{\pi}{4}$. Thus, $Q'$ is positive, and consequently  $Q$ is increasing on $[\frac{\alpha}{2}, \frac{\pi}{4}]$. Therefore, the optimal value of $Q$ will be achieved at zero or some point in the interval $[0, \frac{\alpha}{2}]$. Hence, we confine our analysis of $Q$ to this interval.


Remarkably, we can establish that $Q'$ has at most two zeros in the interval $[0, \frac{\alpha}{2}]$. This can be demonstrated by factorizing $Q'$ as a product of a positive function with a convex function:
$$
Q'(\theta)= \frac{1}{2}\rho \|\vx\|_2^2 \cos\left(\theta+\frac{\pi}{4}\right) L(\theta),
$$
where $L: [0, \frac{\alpha}{2}] \rightarrow \R$ is defined as:
\begin{equation}\label{eq:L}
L(\theta)=\frac{\sin(2\theta-\alpha)}{\cos(\theta+\frac{\pi}{4})}+\frac{2\sqrt{2}}{\rho \|\vx\|_2^2}.
\end{equation}
We proceed to demonstrate that $L$ is convex on the interval $[0, \frac{\alpha}{2}]$.

\begin{lemma}\label{lemma:L}
    For $\rho>0$ and  a nonzero vector $\vx \in \R_{\downarrow}^2$ with $\kappa=\frac{x_2}{x_1} \in [0, 1)$, the following statements for the function $L$ given by \eqref{eq:L} hold:
    \begin{itemize}
        \item[(i)] $L$ is convex on the interval $[0, \frac{\alpha}{2}]$, where $\alpha$ is given in \eqref{eq:alpha}.
        \item[(ii)] $L(0)$ is positive, zero, or negative if $\rho x_1x_2-1$ is negative, zero, or positive, respectively.   $L'(0)$ is nonnegative if $\kappa \le \frac{\sqrt{5}-1}{2}$ and  negative if $\kappa > \frac{\sqrt{5}-1}{2}$.
        \item[(iii)] $L$ has at most two roots on the interval $[0, \frac{\alpha}{2}]$.
    \end{itemize}
\end{lemma}
\begin{proof}
    Item (i). Notice that
\begin{eqnarray*}
L'(\theta)&=&\frac{2\cos(2\theta-\alpha)\cos(\theta+\frac{\pi}{4})+\sin(2\theta-\alpha)\sin(\theta+\frac{\pi}{4})}{\cos^2(\theta+\frac{\pi}{4})},\\
L''(\theta)&=&\frac{\frac{1}{2}\sin(2\theta-\alpha)(\sin(2\theta)-1)+2 \cos\alpha}{\cos^3(\theta+\frac{\pi}{4})}.
\end{eqnarray*}
Since both numerator and denominator of $L''$ are positive, $L''(\theta)>0$ for all $\theta \in [0, \frac{\alpha}{2}]$, hence, $L$ is strictly convex on this interval.

Item (ii). We notice that
$$
L(0)=\frac{2\sqrt{2}}{\rho \|\vx\|_2^2} (1-\rho x_1x_2), \quad
L'(0)=\frac{2\sqrt{2}}{\rho \|\vx\|_2^2} (x_1^2-x_2^2-x_1x_2).
$$
Hence, the statements in item (ii) hold.

Item (iii). We have
$$
L\left(\frac{\alpha}{2}\right)  = \frac{2\sqrt{2}}{\rho \|\vx\|_2^2} >0, \quad
L'\left(\frac{\alpha}{2}\right)  = \frac{2}{\cos(\frac{\alpha}{2}+\frac{\pi}{4})}>0.
$$
Together with the convexity of $L$, and the value of $L(0)$, we know that $L$ has at most two zeros on the interval $[0, \frac{\alpha}{2}]$.
\end{proof}

With these preliminaries, we can now present the solution to  problem~\eqref{w-star} associated with $h_1$ in the following theorem,  which provides the outcome of the $\vw$-step of the WRD procedure for the proximity operator of $h_1$.
\begin{proposition}\label{prop:h1:n=2}
For $\rho>0$ and  a nonzero vector $\vx \in \R_{\downarrow}^2$ with $\kappa=\frac{x_2}{x_1} \in [0, 1)$, let the function $Q$ be given by \eqref{def:Q-h1-R2}, and let the function $L$ be given by \eqref{eq:L}. Define $\alpha$ as in \eqref{eq:alpha}.
Then, the optimal solution $\vw^\star$ to problem \eqref{w-star} is represented as:
$$
\vw^\star = \begin{bmatrix} \cos(\theta^\star) \\ \sin(\theta^\star)  \end{bmatrix},
$$
where $\theta^\star$ is determined as follows:
\begin{itemize}
\item[(i)] Case  $\rho x_1 x_2<1$.  We choose
\begin{equation}\label{eq:Case1-last}
\theta^\star=\left\{
 \begin{array}{ll}
    0, & \hbox{if $\kappa \le \frac{\sqrt{5}-1}{2}$;} \\
    0, & \hbox{if $\frac{\sqrt{5}-1}{2}<\kappa<1$, $L(\theta_0) \ge 0$ with $L'(\theta_0)=0$;} \\
    \arg\min\{Q(\theta): \theta \in \{0, \theta_1\}\}, & \hbox{if $\frac{\sqrt{5}-1}{2}<\kappa<1$, $L(\theta_0) < 0$ with $L'(\theta_0)=0$.}
  \end{array}
\right.
\end{equation}
Here $\theta_1$ is the root of $L$ on the interval $(\theta_0, \frac{\alpha}{2})$.

\item[(ii)] Case  $\rho x_1 x_2=1$.  If $\kappa \le \frac{\sqrt{5}-1}{2}$, we choose $\theta^\star=0$;  Otherwise,  $\theta^\star$ is chosen to be the root of $L$ on $(0, \frac{\alpha}{2})$.

\item[(iii)] Case  $\rho x_1 x_2>1$. $\theta^\star$ is chosen to be the only root of $L$ on the interval $[0, \frac{\alpha}{2}]$.
\end{itemize}

\end{proposition}

\begin{proof}\ \
Case  $\rho x_1 x_2<1$. That is, $L(0)>0$ by Lemma~\ref{lemma:L}. Then $Q'$ has no root if $L'(0)\ge 0$. In this situation, $L$ is positive, so is $Q'$ on  $[0, \frac{\alpha}{2}]$. Hence, we choose $\theta^\star=0$; If $L'(0)< 0$, since $L'\left(\frac{\alpha}{2}\right)>0$, there exists one and only one  point $\theta_0 \in (0, \frac{\alpha}{2})$ such that $L'(\theta_0)=0$. If $L(\theta_0)\ge 0$, $Q'$ has no root, we choose $\theta^\star=0$. If $L(\theta_0)< 0$, then $L$ has a unique root, say $\theta_1$, on the interval $(\theta_0, \frac{\alpha}{2})$. In this situation, we choose $\theta^\star=\arg\min\{Q(\theta): \theta \in \{0, \theta_1\}\}$. All situations are summarized in \eqref{eq:Case1-last}.

Case  $\rho x_1 x_2=1$. That is, $L(0)=0$  by Lemma~\ref{lemma:L}.   If $L'(0)\ge 0$, we choose $\theta^\star=0$. On the other hand, if $L'(0)< 0$, let $\theta_1$ be the only root of $L$ on the open interval $(0, \frac{\alpha}{2})$, then $\theta^\star=\theta_1$.

Case  $\rho x_1 x_2>1$. That is, $L(0)<0$ by Lemma~\ref{lemma:L} again.  Let $\theta_1$  be the only root on the open interval $(0, \frac{\alpha}{2})$. Then, $\theta^\star=\theta_1$, and  $Q$ achieves its global minimum at $\theta^\star$.
\end{proof}

Based on Proposition~\ref{prop:h1:n=2}, the set of $\R_{\downarrow}^2\setminus \{\alpha \ve: \alpha \in \mathbb{R}\}$ is split into three disjoint sets $I_1$, $I_2$, and $I_3$, as follows:
\begin{eqnarray*}
I_1&=&\{(x_1,x_2)\in \R_{\downarrow}^2: x_1>x_2, \rho x_1x_2<1\}\\
I_2&=&\{(x_1,x_2)\in \R_{\downarrow}^2: x_1>x_2, \rho x_1x_2=1\}\\
I_3&=&\{(x_1,x_2)\in \R_{\downarrow}^2: x_1>x_2, \rho x_1x_2>1\}.
\end{eqnarray*}
We further split $I_1$ as the union of $I_{1i}$, $i=1,2,3,4$ and $I_2$ as the union of $I_{2i}$, $i=1,2,3,4$ as follows:
$$
\begin{array}{ll}
I_{11}=\left\{(x_1,x_2)\in I_1: x_1>\sqrt{\frac{2}{\rho}}\right\}&
I_{21}=\left\{(x_1,x_2)\in I_2: x_1>\sqrt{\frac{2}{\rho}}\right\} \\
I_{12}=\left\{(x_1,x_2)\in I_1: x_1=\sqrt{\frac{2}{\rho}}\right\}&
I_{22}=\left\{\left(\sqrt{\frac{2}{\rho}}, \sqrt{\frac{1}{2\rho}}\right)\right\}\\
I_{13}=\left\{(x_1,x_2)\in I_1: \frac{\sqrt{5}+1}{2}x_2 \le x_1<\sqrt{\frac{2}{\rho}}\right\}&
I_{23}=\left\{(x_1,x_2)\in I_2: \sqrt{\frac{\sqrt{5}+1}{2\rho}}\le x_1<\sqrt{\frac{2}{\rho}}\right\} \\
I_{14}=\left\{(x_1,x_2)\in I_1: \frac{\sqrt{5}+1}{2}x_2 >x_1\right\}&
I_{24}=\left\{(x_1,x_2)\in I_2: \sqrt{\frac{1}{\rho}}<x_1<\sqrt{\frac{\sqrt{5}+1}{2\rho}}\right\}
\end{array}
$$

With the given sets, the proximity operator of $h_1$ from the WRD procedure is presented in the next theorem.
\begin{theorem}\label{thm:h1:n=2}
Let $\rho>0$. For $\vx \in I_1 \cup I_2$, we have
$$
\mathrm{prox}_{\frac{1}{\rho} h_1}(\vx)=\left\{
\begin{array}{ll}
\{x_1 \ve_1\}& \hbox{if $\vx \in I_{11}\cup I_{21}$;}  \\
\{\mathbf{0}, \sqrt{\frac{2}{\rho}} \ve_1\}& \hbox{if $\vx \in I_{12}\cup I_{22}$;}  \\
\{\mathbf{0}\} & \hbox{if $\vx \in I_{13}\cup I_{23}$;}  \\
\arg\min\{F(\vz): \vz\in\{\mathbf{0}, \langle \vx, \vw^\star\rangle \vw^\star\}\} & \hbox{if $\vx \in I_{14}\cup I_{24}$,}
\end{array}
\right.
$$
where $\vw^\star$ is from item (i) or item (ii) of Proposition~\ref{prop:h1:n=2}.

For $\vx \in I_3$, we have
$$
\mathrm{prox}_{\frac{1}{\rho} h_1}(\vx)=\arg\min\{F(\vz): \vz\in\{\mathbf{0}, \langle \vx, \vw^\star\rangle \vw^\star\}\},
$$
where $\vw^\star$ is from item (iii) of Proposition~\ref{prop:h1:n=2}.
\end{theorem}
\begin{proof}\ \ The $\vw$-step of the WRD procedure provides $\vw^\star$, the solution of optimization problem~\eqref{w-star} associated with the function $h_1$ by Proposition~\ref{prop:h1:n=2}.  The $r$-step simply follows with $r^\star = \la \vx, \vw^\star \ra$. At the $d$-step, we compare $F(r^\star \vw^\star)$ and $F(\vzero)$ with $F$ defined in \eqref{def:F}. Note that
$$
F(r^\star \vw^\star)-F(\vzero)=G(\vw^\star).
$$
If $G(\vw^\star)$ is positive, the zero is in $\mathrm{prox}_{\frac{1}{\rho} h_1}(\vx)$; if $G(\vw^\star)$ is negative, $r^\star \vw^\star$ is in $\mathrm{prox}_{\frac{1}{\rho} h_1}(\vx)$; if $G(\vw^\star)$ is zero, both the zero vector and $r^\star \vw^\star$ are in $\mathrm{prox}_{\frac{1}{\rho} h_1}(\vx)$. The rest of the result follows directly from Proposition~\ref{prop:h1:n=2}.
\end{proof}

Figure~\ref{fig:Prox_h1_Sparse_Region-2}(a) illustrates the region where the proximity operator $\mathrm{prox}_{\frac{1}{\rho} h_1}$ maps points to the origin.  According to Theorem~\ref{thm:h1-R2-ae},  all points on the line segment from the origin to $(\sqrt{{\sqrt{2}}/{\rho}},\sqrt{{\sqrt{2}}/{\rho}})$ will be mapped to the origin by $\mathrm{prox}_{\frac{1}{\rho} h_1}$. Additionally, as stated in Theorem~\ref{thm:h1:n=2}, all points under the line $x_2=\frac{\sqrt{5}-1}{2}x_1$ in the red region are mapped to the origin by $\mathrm{prox}_{\frac{1}{\rho} h_1}$. The remaining points in both red and blue colors are obtained numerically with the assistance of Theorem~\ref{thm:h1:n=2}.

\begin{figure}[ht]
    \centering
    \begin{tabular}{cc}
    \includegraphics[width=4.5cm]{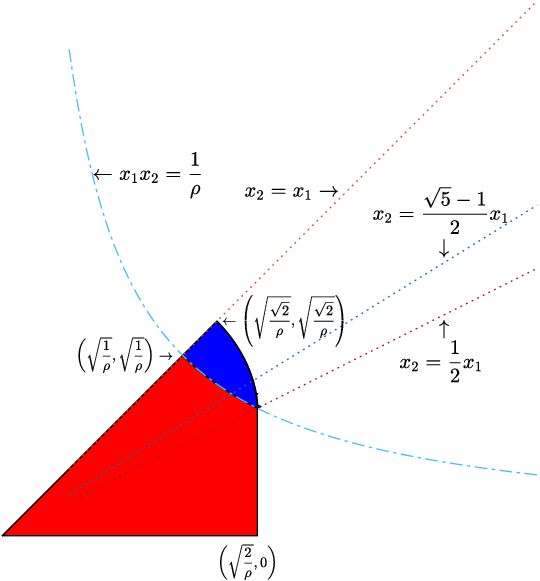}&
    \includegraphics[width=4.5cm]{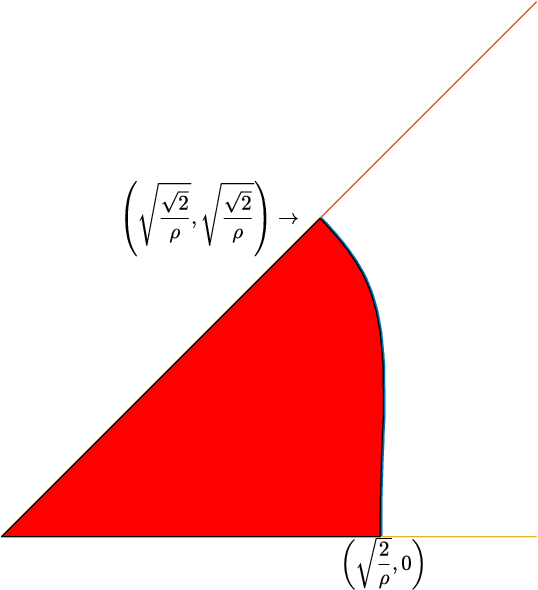}\\

    (a)&(b)
    \end{tabular}
    \caption{(a) The proximity operator $\mathrm{prox}_{\frac{1}{\rho} h_1}$ will map all points in the shaded region to the origin; (b) Numerical result for the region which will be mapped to the origin by the  $\mathrm{prox}_{\frac{1}{\rho} h_1}$.}
    \label{fig:Prox_h1_Sparse_Region-2}
\end{figure}

\subsection{General case: the proximity operator of $h_1$ on $\mathbb{R}^n$}
Here, we demonstrate that if the last $k$ entries of $\vx\in \mathbb{R}^n_\downarrow$ are zero, then the last $k$ entries of $\vw^\star$, the optimal solution to problem \eqref{w-star}, are zero as well. Leveraging this result, we proceed by assuming that all entries of $\vx \in \mathbb{R}^n_\downarrow$ are all nonzero. The primary outcome of this subsection is the transformation of problem \eqref{w-star} into the one with same objective function but constrained on a convex set. The modified problem can be addressed using the nonconvex gradient projection algorithm in \cite{Attouch-Bolte-Svaiter:MP:13}. Subsequently, we introduce an algorithm for computing the proximity operator of $h_1$ on $\mathbb{R}^n$.

\begin{theorem}\label{thm:h1:n>2-xn=0}
For $\rho>0$ and $\vx \in \R_{\downarrow}^n$, suppose that the last $k \ge 1$ entries of $\vx$ are zeros,  that is,
$$
\vx=\begin{bmatrix}
    \vx_{[n-k]} \\ \mathbf{0}
\end{bmatrix},
$$
Then, for an optimal solution $\vw^\star$ to problem \eqref{w-star}, we have $\vw^\star_{[n]\setminus [n-k]}=\mathbf{0}$, that is,  the last $k$ entries of $\vw^\star$ are zero.
\end{theorem}
\begin{proof} \ \ The proof hinges on iteratively reducing the dimension by one up to $k$ steps.  Without loss of generality, we assume that $k=1$. Let $F$ denote the objective function of problem \eqref{w-star} defined on $\mathbb{S}^{n-1}_{+}$. Throughout this proof, we consistently treat $\vw_{[n-1]}$ as the truncation of $\vw$ from its first $(n-1)$ entries.

Define: $H: \mathbb{B}^{n-1}_{+}(\mathbf{0},1) \rightarrow \mathbb{R}$ as $H(\vw_{[n-1]}):=G(\vw)$. Considering the last entry of $\vx$ being zero, we have
\begin{eqnarray*}
H(\vw_{[n-1]})&=&\underbrace{-\frac{1}{2}\rho \langle \vx_{[n-1]}, \vw_{[n-1]}\rangle^2}_{H_1(\vw_{[n-1]})} + \underbrace{\sum_{i=1}^{n-1} w_i +\sqrt{1-\sum_{i=1}^{n-1}w_i^2}}_{H_2(\vw_{[n-1]})}.
\end{eqnarray*}
We can  verify that  both $H_1$ and $H_2$ are concave functions over the domain $\mathbb{B}^{n-1}_{+}$, hence, the minimal value of $H_1+H_2$ will be achieved at $\vw^\star_{[n-1]}$ on the boundary of the ball.


We remark that $\vw^\star_{[n-1]}$ cannot be the zero vector. If so, $H(\vw^\star_{[n-1]})=H(\mathbf{0})=1$. However, $H(\ve_1)=-\frac{1}{2}\rho x_1^2+1<1$, which contradicts $\vw^\star_{[n-1]}$ being the minimal solution to $H$.

Next, we show that $\vw^\star_{[n-1]}$ must be a unit vector, that is, $\vw^\star_{[n-1]} \in \mathbb{S}^{n-2}_{+}$. If not, assume that $\|\vw^\star_{[n-1]}\|_2<1$, we can show that there exists a better solution on the boundary of $\mathbb{B}^{n-1}_{+}(\mathbf{0},1) $, which contradicts the optimality of $\vw_{[n-1]}^\star$. Write $\tilde{\vw}^\star_{[n-1]}=\frac{\vw^\star_{[n-1]}}{\|\vw^\star_{[n-1]}\|_2}$, we define $C:[0,1] \rightarrow \R$ as follows:
$$
C(\lambda):=H_1(\lambda \tilde{\vw}^\star_{[n-1]})+H_2(\lambda \tilde{\vw}^\star_{[n-1]}).
$$
Clearly,
$$
C(\lambda)=-\frac{1}{2}\rho \langle \vx_{[n-1]}, \tilde{\vw}^\star_{[n-1]}\rangle^2 \lambda^2 + \langle \ve, \tilde{\vw}^\star_{[n-1]}\rangle \lambda +\sqrt{1-\lambda^2},
$$
which is not constant, and concave with respect to the variable $\lambda$. Therefore, the minimal value can only be achieved at $\lambda=1$.  Therefore, $\|\vw^\star_{[n-1]}\|_2=1$. In other words, the $n$-th entry of the optimal solution $\vw^\star$ to problem \eqref{w-star} must be $0$. This completes the proof.
\end{proof}

Note that for problem \eqref{w-star}, the feasible set $\Ss_+^{n-1}$ is nonconvex. This nonconvex nature poses significant challenges in algorithm development. To address this, we present the following result which allows us to consider the problem within the confines of a convex set, specifically $\mathbb B_{+}^n(\mathbf{0}, 1)$. This approach provides a more tractable pathway for algorithmic development and analysis.

\begin{theorem}\label{thm:ball-condition:h1}
Let $\vx \in \R_{\downarrow}^n$ and assume that its last entry $x_n$ is nonzero. Let $\vw^\star$ be an optimal solution to the following optimization problem
\begin{equation} \label{eqn:proxrewrite-convex}
    \min\left\{\frac{1}{2}\vw^{\top}\mA_{\rho, \vx} \vw+\ve^{\top} \vw: \vw \in \mathbb{B}_{\downarrow}^n(\mathbf{0},1)\right\},
\end{equation}
where $\mA_{\rho, \vx}$ is given by \eqref{def:A1}. Then, $\vw^\star$ is either the origin or the optimal solution to the optimization problem~\eqref{w-star}. Furthermore, we have
\begin{equation}\label{eq:h1_prox-final}
    \langle \vx, \vw^\star\rangle \vw^\star \in \mathrm{prox}_{\frac{1}{\rho}h_1}(\vx).
\end{equation}
\end{theorem}
\begin{proof}\ \ The proof is trivial if $\vw^\star$ is the zero vector. If $\vw^\star \neq \mathbf{0}$, we now show that $\|\vw^\star\|_2=1$, i.e. $\vw^\star$ is  the optimal solution to the optimization problem~\eqref{w-star}. If not, we denote the objection function of problem~\eqref{eqn:proxrewrite-convex} by $H$, that is,
$$
H(\vw)=\frac{1}{2}\vw^{\top}\mA_{\rho, \vx} \vw+\ve^{\top} \vw.
$$
Set  $\tilde{\vw}^\star: = \frac{\vw^\star}{\|\vw^\star\|_2}$,
and define $C:[0,1] \rightarrow \mathbb{R}$ as follows:
$$
C(\lambda)=H(\lambda \tilde{\vw}^\star)=-\lambda\left(\frac{1}{2}\rho \langle \vx, \tilde{\vw}^\star \rangle^2 \lambda-\|\tilde{\vw}^\star\|_1\right).
$$
Clearly, $C$ achieves its optimal value at either $\lambda=0$ or $1$.    Hence, $$H({\vw}^\star)=C(\|\vw^\star\|_2) > \min\{C(0), C(1)\} = \min\{H(\mathbf{0}), H(\tilde{\vw}^\star)\}.$$
We conclude that $\vw^\star$ is either the origin or the optimal solution to the optimization problem~\eqref{w-star}.

Finally, we show the inclusion \eqref{eq:h1_prox-final} holds. If $\vw^\star=\mathbf{0}$, then, for all $\vw \in \mathbb{S}^{n-1}_{+}$,
$$0=H(\mathbf{0}) \le H(\vw) = G(\vw),
$$
where $G$ is the objective function of problem~\eqref{w-star}. Therefore, no matter which the optimal point for problem~\eqref{w-star} is, we know $\mathbf{0}\in \mathrm{prox}_{\frac{1}{\rho}h_1}(\vx)$.

If $\vw^\star \neq \mathbf{0}$, then $\vw^\star$ is the optimal solution to problem~\eqref{w-star} as well. Hence $\vw^\star$ is the output of the $\vw$-step of the WRD procedure and  $G(\vw^\star)<0$. Obviously,  $\langle \vx, \vw^\star\rangle \vw^\star \in \mathrm{prox}_{\frac{1}{\rho}h_1}(\vx)$ by the $r$-step and $d$-step of the WRD procedure. We conclude that the inclusion \eqref{eq:h1_prox-final} holds.
\end{proof}

Based on Theorem~\ref{thm:ball-condition:h1}, computing $\mathrm{prox}_{\frac{1}{\rho}h_1}(\vx)$ is resorting to solving optimization problem~\eqref{eqn:proxrewrite-convex}. This problem has a concave objective function restricted on a convex set. A popular algorithm for solving problem~\eqref{eqn:proxrewrite-convex} is called nonconvex gradient projection algorithm as follows: with an initial guess $\vw^{(0)}$, iterative
\begin{equation}\label{eq:iterate-w}
    \vw^{(k+1)} = P_{\mathbb{B}_{\downarrow}^n(\mathbf{0},1)}\left(\vw^{(k)}-\frac{1}{2\rho\|\vx\|_2^2}(\mA_{\rho, \vx} \vw^{(k)} + \ve)\right),
\end{equation}
where $P_{\mathbb{B}_{\downarrow}^n(\mathbf{0},1)}$ is the projection operator onto the set $\mathbb{B}_{\downarrow}^n(\mathbf{0},1)$. Since $\mathbb{B}_{\downarrow}^n(\mathbf{0},1)$ is a closed and bounded semi-algebraic convex subset of $\mathbb{R}^n$ and the gradient of the objective function of problem~\eqref{eqn:proxrewrite-convex} is $\mA_{\rho, \vx} \vw + \ve$ with Lipschtiz constant $\rho\|\vx\|_2^2$, the sequence $\{\vw^{(k)}\}_{k \in \mathbb{N}}$ converges, see, for example \cite[Theorem 5.3]{Attouch-Bolte-Svaiter:MP:13}.

We are ready now to present our algorithm for computing  $\prox_{\frac{1}{\rho} h_1}$ based on our WRD procedure for arbitrary $\vx \in \mathbb{R}^n$.  This algorithm is presented in Algorithm~\ref{alg:prox_h1}.
\begin{algorithm}
\begin{algorithmic}[1]
\State \textbf{Input:} Vector $\vx \in \mathbb{R}^n$, parameter $\rho > 0$, and an initial guess $\vw^{(0)}$
\State \textbf{Output:} The proximal operator $\text{prox}_{\frac{1}{\rho} h_1}(\vx)$
\Procedure{}{WRD Procedure}
    \State Sort and convert $\vx$ into $\mathbb{R}^n_{\downarrow}$ via a signed permutation matrix $\mP$.
    \State Trim  $\vx$ if necessary by Theorem~\ref{thm:h1:n>2-xn=0}
    \State {Generate $\vw^{(i)}$ via  \eqref{eq:iterate-w} and denote its limit by $\vw^\star$ \hfill \texttt{($\vw$-step)}}

    \State {Form $\vu \gets \langle \vx, \vw^\star\rangle \vw^\star$ by Theorem~\ref{thm:ball-condition:h1} \hfill \texttt{($r$-step and $d$-step)}}

    \State {Pad $\vu$ with a zero block if necessary by Theorem~\ref{thm:h1:n>2-xn=0}}
    \State {$\vu \leftarrow \mP^{-1}\vu \in \prox_{\frac{1}{\rho} h_1}(\vx)$}
\EndProcedure
\end{algorithmic}
\caption{Computing the Proximal Operator of $h_1$} \label{alg:prox_h1}
\end{algorithm}

Due to the inherent nonconvexity of problem~\eqref{eqn:proxrewrite-convex}, the initial guess provided to any algorithms for this problem significantly influences the quality of the solution obtained. In our simulations, we have observed that choosing $\vw^{(0)}=\alpha \frac{\vx}{\|\vx\|_2}$ with $\alpha \in [\frac{1}{4}, \frac{3}{4}]$ tends to yield satisfactory results. The numerical result  with Algorithm~\ref{alg:prox_h1} in $\mathbb{R}^2$ is shown in Figure~\ref{fig:Prox_h1_Sparse_Region-2}(b). In comparison with Figure~\ref{fig:Prox_h1_Sparse_Region-2}(a), the regions which are identified to be mapped to the origin by  $\mathrm{prox}_{\frac{1}{\rho} h_1}$ are consistent.

\section{Conclusions}\label{sec:conclusions}
This paper addresses the computation of proximity operators of scale and signed permutation invariant functions. By delving into the intrinsic properties of these functions, we introduce a procedure called  WRD, which includes the $\vw$-step, $r$-step, and $d$-step, to effectively handle the computation of proximity operators. Specifically, we conduct a thorough investigation into two specific scale and signed permutation invariant functions: the ratio of $\ell_1/\ell_2$ and its square. For the function $(\ell_1/\ell_2)^2$,  we propose an algorithm capable of explicitly generating its proximity operator through a few straightforward steps. Additionally, for the function $\ell_1/\ell_2$, we devise an efficient algorithm with guaranteed convergence to compute its proximity operator.

In future endeavors, we aim to explore the practical applications of these developed algorithms, particularly in sparse signal recovery and image processing domains.

\backmatter






\section*{Declarations}
\begin{itemize}
    \item The authors declare that they have no conflict of interest.
    \item The work of L. Shen was supported in part by the National
Science Foundation under grant DMS-2208385 and by 2023 and 2024 Air Force Summer Faculty Fellowship Program (SFFP). Any opinions, findings
and conclusions or recommendations expressed in this material are those of the
authors and do not necessarily reflect the views of the National Science Foundation and AFRL (Air Force Research
Laboratory).

\end{itemize}


\begin{thebibliography}{22}
\ifx \bisbn   \undefined \def \bisbn  #1{ISBN #1}\fi
\ifx \binits  \undefined \def \binits#1{#1}\fi
\ifx \bauthor  \undefined \def \bauthor#1{#1}\fi
\ifx \batitle  \undefined \def \batitle#1{#1}\fi
\ifx \bjtitle  \undefined \def \bjtitle#1{#1}\fi
\ifx \bvolume  \undefined \def \bvolume#1{\textbf{#1}}\fi
\ifx \byear  \undefined \def \byear#1{#1}\fi
\ifx \bissue  \undefined \def \bissue#1{#1}\fi
\ifx \bfpage  \undefined \def \bfpage#1{#1}\fi
\ifx \blpage  \undefined \def \blpage #1{#1}\fi
\ifx \burl  \undefined \def \burl#1{\textsf{#1}}\fi
\ifx \doiurl  \undefined \def \doiurl#1{\url{https://doi.org/#1}}\fi
\ifx \betal  \undefined \def \betal{\textit{et al.}}\fi
\ifx \binstitute  \undefined \def \binstitute#1{#1}\fi
\ifx \binstitutionaled  \undefined \def \binstitutionaled#1{#1}\fi
\ifx \bctitle  \undefined \def \bctitle#1{#1}\fi
\ifx \beditor  \undefined \def \beditor#1{#1}\fi
\ifx \bpublisher  \undefined \def \bpublisher#1{#1}\fi
\ifx \bbtitle  \undefined \def \bbtitle#1{#1}\fi
\ifx \bedition  \undefined \def \bedition#1{#1}\fi
\ifx \bseriesno  \undefined \def \bseriesno#1{#1}\fi
\ifx \blocation  \undefined \def \blocation#1{#1}\fi
\ifx \bsertitle  \undefined \def \bsertitle#1{#1}\fi
\ifx \bsnm \undefined \def \bsnm#1{#1}\fi
\ifx \bsuffix \undefined \def \bsuffix#1{#1}\fi
\ifx \bparticle \undefined \def \bparticle#1{#1}\fi
\ifx \barticle \undefined \def \barticle#1{#1}\fi
\bibcommenthead
\ifx \bconfdate \undefined \def \bconfdate #1{#1}\fi
\ifx \botherref \undefined \def \botherref #1{#1}\fi
\ifx \url \undefined \def \url#1{\textsf{#1}}\fi
\ifx \bchapter \undefined \def \bchapter#1{#1}\fi
\ifx \bbook \undefined \def \bbook#1{#1}\fi
\ifx \bcomment \undefined \def \bcomment#1{#1}\fi
\ifx \oauthor \undefined \def \oauthor#1{#1}\fi
\ifx \citeauthoryear \undefined \def \citeauthoryear#1{#1}\fi
\ifx \endbibitem  \undefined \def \endbibitem {}\fi
\ifx \bconflocation  \undefined \def \bconflocation#1{#1}\fi
\ifx \arxivurl  \undefined \def \arxivurl#1{\textsf{#1}}\fi
\csname PreBibitemsHook\endcsname

\bibitem[\protect\citeauthoryear{Candes
  et~al.}{2008}]{Candes-Wakin-Boyd:JFAA:08}
\begin{barticle}
\bauthor{\bsnm{Candes}, \binits{E.}},
\bauthor{\bsnm{Wakin}, \binits{M.B.}},
\bauthor{\bsnm{Boyd}, \binits{S.}}:
\batitle{Enhancing sparsity by reweighted $\ell^1$ minimization}.
\bjtitle{Journal of Fourier Analysis and Applications}
\bvolume{14},
\bfpage{877}--\blpage{905}
(\byear{2008})
\end{barticle}
\endbibitem

\bibitem[\protect\citeauthoryear{Prater-Bennette
  et~al.}{2022}]{Prater-Shen-Tripp:JSC:2022}
\begin{barticle}
\bauthor{\bsnm{Prater-Bennette}, \binits{A.}},
\bauthor{\bsnm{Shen}, \binits{L.}},
\bauthor{\bsnm{Tripp}, \binits{E.E.}}:
\batitle{The proximity operator of the log-sum penalty}.
\bjtitle{Journal of Scientific Computing}
\bvolume{93}(\bissue{3}),
\bfpage{1}--\blpage{34}
(\byear{2022})
\end{barticle}
\endbibitem

\bibitem[\protect\citeauthoryear{Lopes}{2016}]{Lopes:IEEEIT:2016}
\begin{barticle}
\bauthor{\bsnm{Lopes}, \binits{M.E.}}:
\batitle{Unknown sparsity in compressed sensing: Denoising and inference}.
\bjtitle{IEEE Transactions on Information Theory}
\bvolume{62}(\bissue{9}),
\bfpage{5145}--\blpage{5166}
(\byear{2016})
\end{barticle}
\endbibitem

\bibitem[\protect\citeauthoryear{Rahimi
  et~al.}{2019}]{Rahimi-Wang-Dong-Lou:SIAMSC:2019}
\begin{barticle}
\bauthor{\bsnm{Rahimi}, \binits{Y.}},
\bauthor{\bsnm{Wang}, \binits{C.}},
\bauthor{\bsnm{Dong}, \binits{H.}},
\bauthor{\bsnm{Lou}, \binits{Y.}}:
\batitle{A scale-invariant approach for sparse signal recovery}.
\bjtitle{SIAM Journal on Scientific Computing}
\bvolume{41}(\bissue{6}),
\bfpage{3649}--\blpage{3672}
(\byear{2019})
\end{barticle}
\endbibitem

\bibitem[\protect\citeauthoryear{Tang and
  Nehorai}{2011}]{Tang-Nehorai:IEEESP:2011}
\begin{barticle}
\bauthor{\bsnm{Tang}, \binits{G.}},
\bauthor{\bsnm{Nehorai}, \binits{A.}}:
\batitle{Performance analysis of sparse recovery based on constrained minimal
  singular values}.
\bjtitle{IEEE Transactions on Signal Processing}
\bvolume{59}(\bissue{12}),
\bfpage{5734}--\blpage{5745}
(\byear{2011})
\end{barticle}
\endbibitem

\bibitem[\protect\citeauthoryear{Yin et~al.}{2014}]{Yin-Esser-Xin:CIS14}
\begin{barticle}
\bauthor{\bsnm{Yin}, \binits{P.}},
\bauthor{\bsnm{Esser}, \binits{E.}},
\bauthor{\bsnm{Xin}, \binits{J.}}:
\batitle{Ratio and difference of $\ell_1$ and $\ell_2$ norms and sparse
  representation with coherent dictionaries}.
\bjtitle{Communications in Information and Systems}
\bvolume{14}(\bissue{2}),
\bfpage{87}--\blpage{109}
(\byear{2014})
\end{barticle}
\endbibitem

\bibitem[\protect\citeauthoryear{Xu
  et~al.}{2021}]{Xu-Narayan-Tran-Webster:ACHA2021}
\begin{barticle}
\bauthor{\bsnm{Xu}, \binits{Y.}},
\bauthor{\bsnm{Narayan}, \binits{A.}},
\bauthor{\bsnm{Tran}, \binits{H.}},
\bauthor{\bsnm{Webster}, \binits{C.G.}}:
\batitle{Analysis of the ratio of $\ell_1$ and $\ell_2$ norms in compressed
  sensing}.
\bjtitle{Applied and Computational Harmonic Analysis}
\bvolume{55},
\bfpage{486}--\blpage{511}
(\byear{2021})
\end{barticle}
\endbibitem

\bibitem[\protect\citeauthoryear{Attouch
  et~al.}{2013}]{Attouch-Bolte-Svaiter:MP:13}
\begin{barticle}
\bauthor{\bsnm{Attouch}, \binits{H.}},
\bauthor{\bsnm{Bolte}, \binits{J.}},
\bauthor{\bsnm{Svaiter}, \binits{B.F.}}:
\batitle{Convergence of descent methods for semi-algebraic and tame problems:
  proximal algorithms, forward-backward splitting, and regularized
  {Gauss-Seidel} methods}.
\bjtitle{Mathematical Programming, Ser. A}
\bvolume{137},
\bfpage{91}--\blpage{129}
(\byear{2013})
\end{barticle}
\endbibitem

\bibitem[\protect\citeauthoryear{Beck and
  Teboulle}{2009}]{Beck-Teboulle:SIAMIS:09}
\begin{barticle}
\bauthor{\bsnm{Beck}, \binits{A.}},
\bauthor{\bsnm{Teboulle}, \binits{M.}}:
\batitle{A fast iterative shrinkage-thresholding algorithm for linear inverse
  problems}.
\bjtitle{SIAM Journal on Imaging Sciences}
\bvolume{2},
\bfpage{183}--\blpage{202}
(\byear{2009})
\end{barticle}
\endbibitem

\bibitem[\protect\citeauthoryear{Bolte
  et~al.}{2014}]{Bolte-Sabach-Teboulle:MP:2014}
\begin{barticle}
\bauthor{\bsnm{Bolte}, \binits{J.}},
\bauthor{\bsnm{Sabach}, \binits{S.}},
\bauthor{\bsnm{Teboulle}, \binits{M.}}:
\batitle{Proximal alternating linearized minimization for nonconvex and
  nonsmooth problems}.
\bjtitle{Mathematical Programming}
\bvolume{146},
\bfpage{449}--\blpage{494}
(\byear{2014})
\end{barticle}
\endbibitem

\bibitem[\protect\citeauthoryear{Combettes and
  Wajs}{2005}]{Combettes-Wajs:MMS:05}
\begin{barticle}
\bauthor{\bsnm{Combettes}, \binits{P.}},
\bauthor{\bsnm{Wajs}, \binits{V.}}:
\batitle{Signal recovery by proximal forward-backward splitting}.
\bjtitle{Multiscale Modeling and Simulation: A SIAM Interdisciplinary Journal}
\bvolume{4},
\bfpage{1168}--\blpage{1200}
(\byear{2005})
\end{barticle}
\endbibitem

\bibitem[\protect\citeauthoryear{Krol et~al.}{2012}]{Krol-Li-Shen-Xu:IP:2012}
\begin{barticle}
\bauthor{\bsnm{Krol}, \binits{A.}},
\bauthor{\bsnm{Li}, \binits{S.}},
\bauthor{\bsnm{Shen}, \binits{L.}},
\bauthor{\bsnm{Xu}, \binits{Y.}}:
\batitle{Preconditioned alternating projection algorithms for maximum \emph{a
  Posteriori} {ECT} reconstruction}.
\bjtitle{Inverse Problems}
\bvolume{28},
\bfpage{115005}--\blpage{34}
(\byear{2012})
\end{barticle}
\endbibitem

\bibitem[\protect\citeauthoryear{Li et~al.}{2015}]{Li-Shen-Xu-Zhang:AiCM:15}
\begin{barticle}
\bauthor{\bsnm{Li}, \binits{Q.}},
\bauthor{\bsnm{Shen}, \binits{L.}},
\bauthor{\bsnm{Xu}, \binits{Y.}},
\bauthor{\bsnm{Zhang}, \binits{N.}}:
\batitle{Multi-step fixed-point proximity algorithms for solving a class of
  optimization problems arising from image processing}.
\bjtitle{Advances in Computational Mathematics}
\bvolume{41}(\bissue{2}),
\bfpage{387}--\blpage{422}
(\byear{2015})
\end{barticle}
\endbibitem

\bibitem[\protect\citeauthoryear{Micchelli
  et~al.}{2011}]{Micchelli-Shen-Xu:IP-11}
\begin{barticle}
\bauthor{\bsnm{Micchelli}, \binits{C.A.}},
\bauthor{\bsnm{Shen}, \binits{L.}},
\bauthor{\bsnm{Xu}, \binits{Y.}}:
\batitle{Proximity algorithms for image models: Denoising}.
\bjtitle{Inverse Problems}
\bvolume{27},
\bfpage{045009}--\blpage{30}
(\byear{2011})
\end{barticle}
\endbibitem

\bibitem[\protect\citeauthoryear{Parikh and Boyd}{2014}]{Parikh-Boyd:NF-Opt:14}
\begin{barticle}
\bauthor{\bsnm{Parikh}, \binits{N.}},
\bauthor{\bsnm{Boyd}, \binits{S.}}:
\batitle{Proximal algorithms}.
\bjtitle{Foundations and Trends in Optimization}
\bvolume{1},
\bfpage{123}--\blpage{231}
(\byear{2014})
\end{barticle}
\endbibitem

\bibitem[\protect\citeauthoryear{Tao}{2022}]{Tao:SIAMSC-2022}
\begin{barticle}
\bauthor{\bsnm{Tao}, \binits{M.}}:
\batitle{Minimization of $\ell_1$ over $\ell_2$ for sparse signal recovery with
  convergence guarantee}.
\bjtitle{SIAM Journal on Scientific Computing}
\bvolume{44}(\bissue{2}),
\bfpage{770}--\blpage{797}
(\byear{2022})
\end{barticle}
\endbibitem

\bibitem[\protect\citeauthoryear{Shen
  et~al.}{2019}]{Shen-Suter-Tripp:JOTA:2019}
\begin{barticle}
\bauthor{\bsnm{Shen}, \binits{L.}},
\bauthor{\bsnm{Suter}, \binits{B.W.}},
\bauthor{\bsnm{Tripp}, \binits{E.E.}}:
\batitle{Structured sparsity promoting functions}.
\bjtitle{Journal of Optimization Theory and Applications}
\bvolume{183}(\bissue{3}),
\bfpage{386}--\blpage{421}
(\byear{2019})
\end{barticle}
\endbibitem

\bibitem[\protect\citeauthoryear{Moreau}{1962}]{moreau:RASPS:62}
\begin{barticle}
\bauthor{\bsnm{Moreau}, \binits{J.-J.}}:
\batitle{Fonctions convexes duales et points proximaux dans un espace
  hilbertien}.
\bjtitle{C.R. Acad. Sci. Paris S\'{e}r. A Math.}
\bvolume{255},
\bfpage{1897}--\blpage{2899}
(\byear{1962})
\end{barticle}
\endbibitem

\bibitem[\protect\citeauthoryear{Donoho}{1995}]{donoho:ieeeit:95}
\begin{barticle}
\bauthor{\bsnm{Donoho}, \binits{D.}}:
\batitle{De-noising by soft-thresholding}.
\bjtitle{IEEE Transactions on Information Theory}
\bvolume{41},
\bfpage{613}--\blpage{627}
(\byear{1995})
\end{barticle}
\endbibitem

\bibitem[\protect\citeauthoryear{Tao and An}{1996}]{tao1996difference}
\begin{barticle}
\bauthor{\bsnm{Tao}, \binits{P.D.}},
\bauthor{\bsnm{An}, \binits{L.T.H.}}:
\batitle{Difference of convex functions optimization algorithms ({DCA}) for
  globally minimizing nonconvex quadratic forms on {Euclidean} balls and
  spheres}.
\bjtitle{Operations Research Letters}
\bvolume{19}(\bissue{5}),
\bfpage{207}--\blpage{216}
(\byear{1996})
\end{barticle}
\endbibitem

\bibitem[\protect\citeauthoryear{Mart{\'\i}nez}{1994}]{martinez1994local}
\begin{barticle}
\bauthor{\bsnm{Mart{\'\i}nez}, \binits{J.M.}}:
\batitle{Local minimizers of quadratic functions on {Euclidean} balls and
  spheres}.
\bjtitle{SIAM Journal on Optimization}
\bvolume{4}(\bissue{1}),
\bfpage{159}--\blpage{176}
(\byear{1994})
\end{barticle}
\endbibitem

\bibitem[\protect\citeauthoryear{Wang}{2004}]{Wang:AMM:2004}
\begin{barticle}
\bauthor{\bsnm{Wang}, \binits{X.}}:
\batitle{A simple proof of {Descartes}'s rule of signs}.
\bjtitle{The American Mathematical Monthly}
\bvolume{111}(\bissue{6}),
\bfpage{525}--\blpage{526}
(\byear{2004})
\doiurl{10.1080/00029890.2004.11920108}
\end{barticle}
\endbibitem

\end{thebibliography}



\end{document}